\documentclass[reqno]{amsart}

\usepackage[ruled,vlined,norelsize,noend]{algorithm2e}
\usepackage{tikz-cd}
\usepackage{pgfplots}
\usepackage{tikz}
 \usetikzlibrary{calc,angles,quotes}
\usepackage{pgf}
\usepackage[bb=dsserif]{mathalpha} 
\usepackage{bm}
\usepackage{xurl}
\usepackage[all,hyperref,numberbysection]{bi-discrete}
\usepackage[backend=biber,maxbibnames=5,maxalphanames=5,style=alphabetic,bibencoding=utf8,giveninits,url=false,isbn=false,doi=true,eprint=true]{biblatex}
\renewbibmacro{in:}{}
\AtBeginBibliography{\small}

\numberwithin{equation}{section}
\numberwithin{figure}{section}
\numberwithin{table}{section}

\providecommand{\dsig}{\, \mathrm{d}\sigma(x)}
\providecommand{\dx}{\, \mathrm{d}x}
\providecommand{\dxhat}{\, \mathrm{d}\hat{x}}
\providecommand{\tria}{\mathcal{T}}
\providecommand{\vertices}{\mathcal{V}}
\providecommand{\edges}{\mathcal{E}}
\providecommand{\facets}{\mathcal{F}}
\providecommand{\tr}{\operatorname{tr}}

\providecommand{\interior}{\operatorname{int}}
\providecommand{\support}{\operatorname{supp}}
\providecommand{\BF}{{\operatorname{B_{\mathcal{F}}}}}
\providecommand{\BFp}{{\operatorname{B^\partial_{\mathcal{F}}}}}
\def\N{\mathbb{N}}
\def\Lpmeanfree{L_0^p}

\def\R{\mathbb{R}}

\providecommand{\ohz}{\omega_z}
\providecommand{\Ohz}{\Omega_z}
\providecommand{\Ehz}{\edges_z}
\providecommand{\Wpn}{{W^{1,p}_n(\Omega)}}
\providecommand{\Wpsim}{{W^{1,p}_\sim(\Omega)}}
\renewcommand{\identity}{I}
\newcommand{\infsupdiscr}{{\beta_*}}

\usepackage{todonotes}

\usepackage{a4wide}
\definecolor{cobalt}{rgb}{0.0, 0.28, 0.67}
\definecolor{darkcerulean}{rgb}{0.03, 0.27, 0.49}
\hypersetup{citecolor=darkcerulean,linkcolor=darkcerulean,    urlcolor=darkcerulean}
\usepackage{enumitem}
\usepackage{listings}
\usepackage[thinlines]{easytable}
\usepackage[font=footnotesize,margin=0.2cm]{caption}
\usepackage{subcaption}
\begin{document}

\title[Construction of Fortin operators]{Construction of trace-preserving Fortin operators
}

\author[Eickmann]{Franziska Eickmann}
\author[Tscherpel]{Tabea Tscherpel}

\address[F.\ Eickmann]{Department of Mathematics, Technische Universit\"at Darmstadt, Dolivostraße 15, 64293 Darmstadt, Germany}
\email{eickmann@mathematik.tu-darmstadt.de}

\address[T.\ Tscherpel]{Department of Mathematics, Technische Universit\"at Darmstadt, Dolivostraße 15, 64293 Darmstadt, Germany}
\email{tscherpel@mathematik.tu-darmstadt.de}

\subjclass[2020]{
	65N12,  
	65N15, 
	65N30, 
	76D07,   
	76M10,
	35R37
}

\date{\today}

\maketitle

\begin{abstract}
We present a unifying framework to construct local Fortin operators for conforming mixed finite element pairs for the Stokes equations. 
The operators are constructed to satisfy the divergence-preservation property, local stability and approximation properties, and certain trace-preservation properties. 
	For the latter, some of the finite element pairs require an enrichment of the velocity space. 
	We present the construction for the $P_d-P_0$ element, the Bernardi--Raugel element, the conforming Crouzeix--Raviart element, a modified MINI element and generalised Taylor--Hood elements. 
	Furthermore, we discuss implications of the existence of such a Fortin operator beyond inf-sup stability. 
	These include applications to non-Newtonian fluid flow, problems with inhomogeneous Dirichlet boundary conditions, as well as uniform inf-sup stability relevant for domain approximation and moving domains. 
\end{abstract}

\textbf{Keywords.} Fortin operator, construction, inf-sup stability, approximation properties, domain approximation 



\section{Introduction}

In the design of mixed finite element methods for the Stokes equations and related incompressible flow problems, inf-sup stability of the pair of discrete velocity and pressure spaces is key for stability and error estimates.  
By the Fortin lemma~\cite{Fortin1977}, provided that the continuous inf-sup condition holds, the discrete inf-sup condition in Hilbert spaces is equivalent to the existence of a bounded linear Fortin operator, that is, an operator that preserves the discrete divergence when tested against functions in the discrete pressure space. 
This is referred to as \emph{divergence-preservation property} in the following. 
This equivalence was extended to Banach spaces in~\cite{Ern2016}, where the operator need not be linear. 
While Fortin operators offer one route to establish inf-sup stability, the converse direction implies that a Fortin operator exists whenever inf-sup stability has been established by some other strategy. 
These abstract results, however, do not in general yield a single operator that is simultaneously stable in~$W^{1,p}(\Omega)$ for all~$p \in (1,\infty)$ and preserves the various trace conditions. 
Instead, for fixed $p \in (1, \infty)$ they guarantee a Fortin operator only  for the particular  homogeneous trace condition associated with the underlying inf-sup condition. 

For many applications, however,  further properties of the Fortin operator are essential, including locality, stability with respect to several $W^{1,p}$-norms and preservation of more general trace conditions. 
Such properties are not provided by the abstract Fortin lemma. 
Consequently, a number of works have developed local constructions of Fortin operators for various $H^1$-conforming finite element pairs, see e.g.~\cite{ArnoldBrezziFortin1984,BernardiRaugel1985,GiraultLions2001,Girault2003,Falk2008,MardalSchoeberlWinther2013,DieningStornTscherpel2022,Scott2021,Parker2025,EickmannGuzmanNeilanEtAl2026} for an incomplete list.
Most of them focus on velocity functions with zero traces, as relevant for fluid problems with homogeneous Dirichlet boundary conditions.

In applications with inhomogeneous Dirichlet or more general boundary conditions, Fortin operators acting on functions with nonzero traces are needed. 
In this setting, one requires not only divergence preservation in the dual of the discrete pressure space, but also suitable trace-preservation properties, possibly restricted to a part of the boundary~\cite[Sec.~6.1.1.6]{Bernardi2024}.  
If the problem is nonlinear, then locality and stability with respect to $W^{1,p}(\Omega)$ for all~$p\in(1,\infty)$ are key features. 
Fortin operators with such properties are an important tool in the numerical analysis of incompressible fluids, for example for non-Newtonian fluids~\cite{BelenkiBerselliDieningEtAl2012,DieningKreuzerSueli2013,SueliTscherpel2020,Jessberger2024} and for general boundary conditions~\cite{Eickmann2025a,	GazcaOrozco2025}. 
We demonstrate how they can be applied to show uniform inf-sup stability, which plays a role in domain approximation~\cite{BaeKim2004,GjerdeScott2022},  moving-domain problems~\cite{BurmanFreiMassing2022,WahlRichterLehrenfeld2022,NeilanOlshanskii2024}, and in fluid-structure interaction~\cite{FreiKnokeSteinbachEtAl2026}. 
In particular, uniform inf-sup stability is required in the analysis of  XFEM~\cite{KirchhartGrossReusken2016} and cutFEM methods~\cite{BurmanClausHansboEtAl2015,BurmanHansboLarsonEtAl2025}. Uniform inf-sup stability has previously been established in~\cite{KirchhartGrossReusken2016,GuzmanOlshanskii2018} using a different approach.
For exactly divergence-free elements, preserving the mean of the normal traces provides a direct way, in particular, to enforce compatibility of the discrete boundary data~\cite{Eickmann2025a}. 

Many constructions of Fortin operators with additional properties proceed by successive correction. 
This means that, starting from a quasi-interpolation operator~$I_h$, one introduces a divergence-correcting operator~$C_h$ and defines the Fortin operator~$\Pi_h$ by 
\begin{align}\label{eq:Fortin-correction}
	(I-\Pi_h)=(I-C_h)(I-I_h),
\end{align}
where~$I$ denotes the identity. 
If both~$I_h$ and~$C_h$ are local, then so is~$\Pi_h$.    
The  construction of~$C_h$ depends on whether the pressure functions are continuous or discontinuous. 
For continuous pressure functions, global integration by parts is available. 
Existing constructions for, e.g., the MINI~\cite{ArnoldBrezziFortin1984,BelenkiBerselliDieningEtAl2012} and the (generalised) Taylor--Hood~\cite{Girault2003,DieningStornTscherpel2022} elements rely on vanishing boundary traces and do not account for trace preservation.  
For discontinuous pressure functions, the discrete divergence preservation is more local, which simplifies the construction of~$\Pi_h$. 
In this case, integration by parts must be applied locally on each simplex, resulting in facet terms involving normal velocity traces. 
These terms are analogous to boundary terms. 
For this reason, standard constructions, e.g., for the Bernardi--Raugel element~\cite{BernardiRaugel1985}, the $P_2-P_0$ element in two dimensions~\cite[Ex.~1]{CrouzeixRaviart1973},~\cite[Sec.~8.4.3]{Boffi2013} and its generalisation $P_d-P_0$ (for dimension $d$), the conforming Crouzeix--Raviart element~\cite{CrouzeixRaviart1973,Girault2003} and the Guzmán--Neilan element~\cite{Guzman2014a,Guzman2014b} naturally extend to trace-preserving operators, although they are often only formulated for homogeneous traces. 
For the Bernardi--Raugel element, for example, using a Scott--Zhang-type quasi-interpolation operator~\cite{ScottZhang1990} based on the degrees of freedom of the element within the construction in~\cite{BernardiRaugel1985} yields a Fortin operator with the desired properties. 
A similar approach can be employed for the conforming Crouzeix--Raviart element~\cite{CrouzeixRaviart1973} if the degrees of freedom are chosen suitably.

In this work, we extend the correction approach~\eqref{eq:Fortin-correction} to a general framework for constructing local, stable, and trace-preserving Fortin operators for a range of $H^1$-conforming finite element spaces. 
We split the operator~$C_h$ into a trace correction, which accounts for boundary and, where needed, interior facet traces, and a divergence correction. 
For discontinuous pressures, this yields a natural extension of standard constructions, whereas elements with continuous pressures require a modification of the correction operators.  
In some low-order cases, this is achieved by enriching the velocity space with suitable boundary facet bubbles. 
We formulate the required properties of  correction operators in an abstract setting where one only has to verify suitable local stability  and divergence-preservation properties, see Proposition~\ref{prop:Chj_operators}. 
This provides a unified framework for existing constructions and for the design of new ones.

Our framework covers conforming finite element pairs, including the Bernardi--Raugel, $P_d-P_0$, conforming Crouzeix--Raviart, (modified) MINI, and Taylor--Hood elements as well as their higher-order extensions. 
An overview of the finite element pairs considered and the properties established is presented in Table~\ref{tbl:fortin-summary}.  
Local trace-preserving Fortin operators for the Scott--Vogelius element in two dimensions have also recently been constructed in~\cite{Parker2025,EickmannGuzmanNeilanEtAl2026}. 
For trace-preserving Fortin operators for three-dimensional non-conforming finite elements based on a de~Rham complex, we refer to~\cite{Ern2026}. 

For finite element pairs with continuous pressure spaces, the trace-preserving framework is, to the best of our knowledge, new. 
In addition, we recast existing constructions in a more constructive form, facilitating further analysis of their properties, in particular for the higher-order Taylor--Hood elements. 
The correction operators are closely related to the local macro-element approach in~\cite{Girault2003}, which is based on the technique originally introduced in~\cite{Stenberg1984}. 
Within our framework, we provide a constructive formulation of this 
approach, and we extend the construction in~\cite{Girault2003,DieningStornTscherpel2022} to trace-preserving operators. 

In addition, we collect several consequences and applications of Fortin operators with these additional properties. 
We review applications to non-Newtonian fluid problems and to problems with inhomogeneous Dirichlet boundary conditions. 
Furthermore, in Theorem~\ref{thm:domain approximation}, we establish uniform discrete inf-sup stability for $p \in (1,\infty)$ in the context of domain approximation and moving-domain problems. 
The proof uses the standard approach presented, e.g., in~\cite[Prop.~5.4.2]{Boffi2013}. 
The argument relies on locality of the Fortin operator and (uniform) continuous inf-sup stability on sequences of John domains, as established in \cite{AcostaDuranMuschietti2006,DieningRuzickaSchumacher2010}, see also~\cite{JiangKauranenKoskela2014}.  
The assumptions on the discrete domain differ from those in~\cite[Thm.~1]{GuzmanOlshanskii2018}, where uniform discrete inf-sup stability for~$p = 2$ is established  for a range of mixed finite element pairs. 
The discussion is intended to illustrate the relevance of the framework rather than to provide an exhaustive account of its applications. 

Finally, we present the construction of Fortin operators for general inf-sup stable finite element pairs. 
The resulting operator is not local,  but it satisfies trace-preservation properties.

\subsubsection*{Outline} In Section~\ref{sec:preliminaries}, we introduce the notation for the relevant function spaces and review preliminaries on mixed finite element pairs. 
Section~\ref{sec:main_results} summarises the main results on constructive Fortin operators and their properties for a range of mixed finite element pairs, and Section~\ref{sec:applications} discusses some of the implications. 
In Section~\ref{sec:localFortin}, we present the local construction principle for Fortin operators. 
We consider finite element pairs with discontinuous pressure spaces in Subsection~\ref{subsec:discont}, including Bernardi--Raugel, $P_d-P_0$, and conforming Crouzeix--Raviart elements.
Pairs with continuous pressure spaces, such as the MINI and Taylor--Hood elements, are addressed in Subsection~\ref{subsec:cont}. 
Finally, in Section~\ref{sec:globalFortin} we discuss a general but global construction of a trace-preserving Fortin operator based on the existence of a Fortin operator for the homogeneous subspaces, as provided by the Fortin lemma.


\section{Preliminaries on mixed finite element methods}\label{sec:preliminaries}
We start by introducing notation and preliminaries on the finite element spaces used in the following. 

The spectral norm of a matrix~$A\in\R^{m\times m}$ with~$m \in \mathbb{N}$  is denoted by~$\norm{A}_2$. 
Furthermore, $\abs{\cdot}$ denotes the Euclidean norm in $\setR^m$ for $m \in \mathbb{N}$.   
We also use it for the $d$-dimensional Lebesgue measure of a subset of $\setR^d$, and for the cardinality of discrete sets. 	

In general, for a function space~$X$ we denote by~$X^d$ the corresponding space of vector-valued functions, for $d \in \mathbb{N}$. 
Moreover, $\identity$ denotes the identity operator.  
For~$p\in [1,\infty]$ and a bounded domain~$M \subset \setR^d$ we denote by~$L^p(M)$ the standard Lebesgue space equipped with norm~$\norm{\cdot}_{L^p(M)}$. 
Denoting the integral mean over~$M$  by 
\begin{align*}
\langle f \rangle_M \coloneqq \dashint_{M} f \dx \coloneqq \abs{M}^{-1} \int_{M} f \dx \qquad \text{ for } f \in L^1(M),
\end{align*}
 the subspace of functions in~$L^p(M)$ with vanishing integral mean is defined as 
\begin{align} \label{eq:def-Lmeanfre}
	\Lpmeanfree(M) &\coloneqq \left\{f\in L^p(M) \colon {\langle f\rangle_M} = 0\right\}. 
\end{align}
Let~$\skp{\cdot}{\cdot}_M$ denote the scalar product on~$L^2(M)$. 
We also use this notation for the pairing of functions in~$L^p(M)$  and $L^{p'}(M)$, for $p \in [1,\infty]$, where  $p'\in [1,\infty]$ is the Hölder conjugate determined by $\tfrac1p + \tfrac{1}{p'}=1$. 
We drop the index $M$ if there is no risk of ambiguity.  

For~$d \geq 2$ let~$\Omega\subseteq \R^d$ be a bounded Lipschitz domain with polyhedral boundary.
We denote by~$W^{1,p}(\Omega)$ for~$p\in[1,\infty]$ the standard Sobolev space on~$\Omega$ with norm~$\norm{\cdot}_{W^{1,p}(\Omega)}$. 
Furthermore, for $p \in [1,\infty]$ we define 
\begin{align}\label{def:W1p0}
	W^{1,p}_0(\Omega)^d = \{ v \in W^{1,p}(\Omega)^d \colon \tr(v) = 0 \text{ on } \partial \Omega\}
\end{align}
with traces in $L^p(\partial \Omega)^d$. 

For $p \in [1,\infty]$ we denote the subspace of functions with zero normal traces and the subspace of functions with zero mean of the normal traces by
\begin{align}
	\Wpn 
	&\coloneqq
	 \{v\in W^{1,p}(\Omega)^d\colon  \tr(v)\cdot n = 0 \text{ on }\partial\Omega \},\label{def:W1pn}\\
	\Wpsim
	 &\coloneqq
	  \left\{v\in W^{1,p}(\Omega)^d\colon  \int_{\partial\Omega} \tr(v)\cdot n \dsig = 0\right\},\label{def:W1psim}
\end{align}
where~$n$ denotes the outer unit normal on~$\partial\Omega$. 
Since~$\Omega$ is a bounded Lipschitz domain, the continuous inf-sup condition holds; see~\cite{Bogovskii1979,DieningRuzickaSchumacher2010}. 
Together with the inclusions~$W^{1,p}_0(\Omega)^d \subset \Wpn \subset  \Wpsim  $ and Poincaré's inequality, this yields the following inf-sup estimates for~$p \in (1,\infty)$: There exists a constant~$\beta = \beta(p)>0$ such that 
\begin{equation}\label{def:inf-sup-cont}
\begin{aligned}
	 \beta \norm{q}_{L^{p'}(\Omega)}
	&\leq  \sup_{v \in W^{1,p}_{0}(\Omega)^d \setminus \{0\}} \frac{\skp{\divergence v}{q}}{\norm{\nabla v}_{L^p(\Omega)}} \\
	&\lesssim  
	\sup_{v \in W^{1,p}_{n}(\Omega) \setminus \{0\}} \frac{\skp{\divergence v}{q}}{\norm{ v}_{W^{1,p}(\Omega)}} \leq 
	\sup_{v \in W^{1,p}_{\sim}(\Omega) \setminus \{0\}} \frac{\skp{\divergence v}{q}}{\norm{ v}_{W^{1,p}(\Omega)}}, 
\end{aligned}
\end{equation}
for all $q\in L_0^{p'}(\Omega)$. 

We denote by~$c>0$ a generic constant which may change from line to line.  
The notation~$A \lesssim B$ means that there is a constant $c>0$   depending only on the quantities specified such that $A \leq c B$. 
Similarly, we write~$A \eqsim B$ if~$A \lesssim B$ and $B \lesssim A$.      

Let us collect the notation and the setup for mixed finite element methods for incompressible fluid equations. 

\subsection*{Triangulations} 
Let~$(\tria_n)_{n \in\N}$ be a family of conforming triangulations of $\overline{\Omega} \subset \setR^d$ consisting of closed~$d$-simplices, see~\cite[Ch.~2, \S~2.1]{Ciarlet2002}. 
Let~$h_n \coloneqq \max_{T\in \tria_n} h_T>0$ denote the maximal mesh size, where~$h_T\coloneqq \diameter(T)$ for all~$T\in \tria_n$. 
For simplicity, in the following we suppress the index~$n$.
Instead, with slight abuse of notation we use~$h$ both to denote the maximal mesh size and to parametrise the triangulations.  
We refer to the resulting sequence of triangulations as~$(\tria_h)_{h}$. 
We define the set of~$(d-1)$-dimensional subsimplices (facets) in~$\tria_h$ as~$\facets(\tria_h)$ and the subset of boundary facets as~$\facets^\partial(\tria_h)\coloneqq \{f\in\facets(\tria_h) \colon f\subset\partial\Omega\}$.
Let~$\vertices(\tria_h)$ denote the set of vertices of the triangulation~$\tria_h$, and for a subset of vertices~$\{z_1, \ldots, z_{\ell}\} \subset \vertices(\tria_h)$, with  {$\ell\in\N$}, we denote their convex hull by~$[z_1, \ldots, z_{\ell}]$. 
We use the same notation also for vertices in a collection of subsimplices. 
By~$\vertices^\circ(\tria_h) \coloneqq \{z \in \vertices(\tria_h) \colon z \notin \partial \Omega\}$ we denote the set of interior vertices of $\tria_h$. We write~$\gamma_T\in\N_0$ for the number of vertices of a simplex~$T\in\tria_h$ that lie in $\Omega$ rather than on~$\partial \Omega$. 
Let~$\edges(\tria_h)$ denote the set of $1$-dimensional subsimplices (edges) in~$\tria_h$. 
For a given vertex~$z \in\vertices(\tria_h)$ let $\Ehz \coloneqq \edges_h(z) \subset \edges(\tria_h)$ denote the set of all edges containing~$z$. 

We denote by~$\rho_T>0$ the radius of the largest ball inscribed in~$T\in\tria_h$. 
We assume that the family of triangulations~$(\tria_h)_{h}$ is \emph{shape-regular}, in the sense that there is a constant~$\chi>0$ such that 
\begin{align}\label{eq:shape-regularity}
	\rho_T \geq \chi h_T \qquad \text{for all } T\in \tria_h \; \text{ uniformly in } h.
\end{align}
For a collection of simplices~${\mathcal{S}_h} \subseteq \tria_h$ we denote its neighbourhood and the domain covered by the neighbourhood by 
\begin{align}
	\omega_h({\mathcal{S}_h})&\coloneqq \{T'\in\tria_h\colon~T'\cap T \neq \emptyset~ \text{for some }T\in{\mathcal{S}_h}\},
	\label{def:neighbourhood_S}
	\\
	\Omega_h({\mathcal{S}_h})&\coloneqq \interior \left(\cup \{T\in\omega_h({\mathcal{S}_h})\}\right) \subseteq \Omega.
	\label{def:neighbourhood-domain_S}
\end{align}
For~$s\in\N_0$ the~$s$-neighbourhood of~${\mathcal{S}_h}$ is defined inductively by 
\begin{align}\label{def:neighbourhood_S-level_s}
	\omega_h^0({\mathcal{S}_h})\coloneqq {\mathcal{S}_h}, \qquad \text{ and } \qquad  \omega_h^s({\mathcal{S}_h}) \coloneqq \omega_h(\omega_h^{s-1}({\mathcal{S}_h})) \qquad \text{for } s\in\N.
\end{align}
Let~$\Omega_h^s({\mathcal{S}_h})\coloneqq \interior \left(\cup\{T\in\omega_h^s({\mathcal{S}_h})\}\right)\subseteq \Omega$ denote the domain covered by the~$s$-neighbourhood of ${\mathcal{S}_h}$.
Analogously, for a facet, edge, or vertex~$\sigma \in\facets(\tria_h) \cup \edges(\tria_h) \cup \vertices(\tria_h)$ we denote 
\begin{alignat}{4}
	\omega_h(\sigma) &\coloneqq \{T\in\tria_h\colon  \sigma\subset T\},\qquad &\Omega_h(\sigma) &\coloneqq \interior (\cup \{T\in\omega_h(\sigma)\}).
	\label{def:neighbourhood_f}
\end{alignat}
For simplicity we also use $\ohz \coloneqq \omega_h(z)$ and $\Ohz \coloneqq \Omega_h(z)$ for $z \in \vertices(\tria_h)$. 

Let $\hat{T}$ denote the reference $d$-simplex in $\R^d$. 
For a given simplex $T \in \tria_h$ let the affine transformation 
\begin{align}\label{def:transform}
	\Phi_T\colon \hat{T}\to T,\qquad\hat{x}\mapsto A\hat{x} + a
\end{align}
be determined by $A\in\R^{d\times d}$ and $a\in\R^d$.  
For later use we recall from~\cite[Ch.~3, \S~3.1]{Ciarlet2002} the following standard scaling properties: 
\begin{equation}\label{eq:scaling}
	\begin{aligned}
		&|\det A|\eqsim h_T^d, \quad \norm{A}_2\eqsim h_T, \quad  \norm{A^{-1}}_2\eqsim h_T^{-1}, \quad \dx = |\det A| \dxhat,\\
		&\nabla \psi (x) = A^{-\top} \hat{\nabla}\hat{\psi}(\hat{x})\quad\text{with}\quad\hat{\psi}=\psi\circ \Phi_T \quad \text{for all }\psi\in W^{1,1}(T).
	\end{aligned}
\end{equation}

\subsection*{Mixed finite element spaces}
Let us introduce the conforming velocity and pressure finite element spaces~$X_h\subseteq W^{1,\infty}(\Omega)^d \subset  W^{1,2}(\Omega)^d$ and~$Q_h  \subseteq L^{\infty}(\Omega) \subseteq L^2(\Omega)$, respectively, with respect to a triangulation~$\tria_h$.
For finite-dimensional subspaces~$X(T)$ of smooth vector-valued functions and~$Q(T)$ of smooth scalar-valued functions on~$T\in\tria_h$, we consider 
\begin{align}
	X_h\coloneqq X_h(\tria_h) &\coloneqq \{v\in C(\overline{\Omega})^d\colon v\big|_T \in X(T)~\text{for all } T\in \tria_h \}, 
	 \label{def:Xh} \\
	Q_h \coloneqq Q_h(\tria_h) & \coloneqq \{q\in \mathcal{Q}\colon q\big|_T \in Q(T)~\text{for all }  T\in \tria_h\},
	\label{def:Qh}
\end{align}
where either~$\mathcal{Q}=C(\overline{\Omega})$ or~$\mathcal{Q}=L^\infty(\Omega)$, depending on the concrete finite element pair. 
The subspace of pressure functions with zero integral mean is denoted by 
\begin{align}\label{def:Qhbar}
	\overline{Q}_h\coloneqq Q_h\cap L^{2}_0(\Omega).
\end{align}
The vector-valued Lagrange space arises from \eqref{def:Xh} by choosing $X(T)$ as the space $\mathcal{P}_k(T)^d$ of vector-valued polynomials of degree at most~$k\in \N$ on each simplex $T \in \tria_h$, and we denote it by
\begin{align*}
	\mathcal{L}_k^1(\tria_h)^d \coloneqq \{v\in C(\overline{\Omega})^d\colon v\big|_T \in \mathcal{P}_k(T)^d~\forall T\in \tria_h\}.
\end{align*}
For a subset~$\mathcal{S}_h \subset \tria_h$ and  $S \coloneqq \cup_{T \in \mathcal{S}_h} T$ as in~\eqref{def:neighbourhood-domain_S} we denote 
	\begin{align*}
	\mathcal{L}_{k,0}^1(\mathcal{S}_h)^d \coloneqq \{v\in C(S)^d\colon v\big|_T \in \mathcal{P}_k(T)^d~\forall T\in \mathcal{S}_h, \, v\big|_{\partial S} = 0 \}.
	\end{align*}
Analogously, we denote by~$\mathcal{P}_k(f)^d$ the space of vector-valued polynomials of degree at most~$k\in\N$ on a facet~$f\in\facets(\tria_h)$. 
The specific finite element spaces shall be introduced in Section~\ref{sec:localFortin} below.

With the spaces~$W^{1,p}_0(\Omega)^d$, $\Wpn$ and $\Wpsim$  as defined in~\eqref{def:W1p0}--\eqref{def:W1psim}, we define 
\begin{align}\label{def:Xhstar}
		X_{h,0}\coloneqq X_h\cap W^{1,\infty}_{0}(\Omega)^d, \quad 
			X_{h,n}\coloneqq X_h\cap W^{1,\infty}_{n}(\Omega), \quad \text{ and } \quad 
				X_{h,\sim}\coloneqq X_h\cap W^{1,\infty}_{\sim}(\Omega).
\end{align}
This means we have the following inclusions
\begin{alignat}{6} \notag
	&W^{1,\infty}_{0}(\Omega)^d& \quad\subset\quad & W^{1,\infty}_{n}(\Omega) &\quad\subset\quad&   W^{1,\infty}_{\sim}(\Omega) & \quad\subset\quad& W^{1,\infty}(\Omega)^d  \\
	&\quad \cup &\quad\;\;\;\quad& \quad \cup &\quad\;\;\quad& \quad \cup &\quad\;\;\quad&\quad  \cup \label{eq:subspace-relations}\\
	&\quad X_{h,0} &\quad\subset\quad& \quad X_{h,n} &\quad\subset\quad& \quad X_{h,\sim} &\quad\subset\quad&\quad X_h.  \notag
\end{alignat}

\subsection*{Bubble function spaces}

We consider the hat functions $(\varphi_z)_{z \in \vertices(\tria_h)} \subset \mathcal{L}^1_1(\tria_h)$, defined by 
\begin{align}\label{def:hatfunc}
	\varphi_z(z')  = \delta_{z,z'} \qquad \text{ for any } z,z' \in \vertices(\tria_h). 
\end{align}  
For edges, facets and~$d$-simplices~$\sigma \in \edges(\tria_h) \cup \facets(\tria_h) \cup  \tria_h$ of dimension $d_{\sigma} \in \{1,d-1,d\}$  we define the corresponding scalar bubble function 
\begin{align*}
	b_{\sigma} \coloneqq \prod_{z \in \vertices ( \{\sigma\} ) } \varphi_z \in \mathcal{L}^1_{d_{\sigma} + 1}(\tria_h), 
\end{align*}
with support $\support(b_T) = T$ for $T \in \tria_h$ and $\support(b_{\sigma}) = \overline{\Omega_h(\sigma)}$ for $\sigma \in \edges(\tria_h) \cup \facets(\tria_h)$.  

Let~$\tau_e$ denote one fixed unit tangent vector of an edge~$e \in \edges(\tria_h)$. 
For a facet~$f\in\facets(\tria_h)$ let $n_f$ be one fixed unit normal of the facet. 
Let~$\lambda_0, \ldots, \lambda_d$ be the barycentric coordinates of~$T \in \tria_h$. 
Then a basis of $\mathcal{P}_{k}(T)$ {for~$k\in\N_0$} is given by 
$\{\lambda^{\alpha} \colon  \alpha \in \mathbb{N}_0^{d+1}, \abs{\alpha} = k\}$.      
  
We also need the tangential edge bubble functions~$b_{e}^\tau$, for $e \in \edges(\tria_h)$, the higher-order facet bubble functions~$(b_f^{(i)})_{i \in \{0, \ldots, d-1\}}$, for $f \in \facets(\tria_h)$ and the higher-order element bubble functions $(b_{T}^{\alpha})_{\alpha \in \mathbb{N}_0^{d+1} \colon \abs{\alpha} = k} $  for $T \in \tria_h$ {and~$k\in\N_0$}, defined by 
\begin{alignat}{5}
		b_e^\tau &\coloneqq b_e \tau_e  &&\in \mathcal{L}^1_{2}(\tria_h)^{d}, \label{def:tan_edge_bubble}\\
	b_f^{(i)} & \coloneqq b_f \varphi_{z_i} &&\in \mathcal{L}^1_{d+1}(\tria_h) \qquad &&\text{ for } z_i \in \vertices(\{f\}), \; i \in \{0,\ldots, d-1\}, \label{def:face_bubble_higherdeg}\\
	b_{T}^{\alpha} & \coloneqq b_T \lambda^{\alpha} &&\in \mathcal{L}^{1}_{k + d + 1}(\tria_h) \qquad &&\text{ for } \alpha \in \mathbb{N}_0^{d+1} \colon \abs{\alpha} = k.
\end{alignat}
For $T\in\tria_h$ we define the following local bubble function spaces
\begin{align}
	B(T)&\coloneqq \linearspan \left\{b_T\big|_T\right\} \subseteq \mathcal{P}_{d+1}(T),\label{def:local_element_bubblespace}\\
	B_n(T) &\coloneqq \linearspan\left\{ b_f n_f\big|_T \colon  f\in\facets(\{T\}) \right\} \subseteq \mathcal{P}_d(T)^d, \label{def:local_normal_face_bubblespace}\\
	\BF(T) &\coloneqq \linearspan\left\{ b_f^{(i)} n_f\big|_T \colon  f
	\in\facets(\{T\}),~i\in\{0,...,d-1\}\right\} \label{def:local_face_bubblespace_d+1}  \subseteq \mathcal{P}_{d+1}(T)^d,\\
	B_{k,d}(T) &\coloneqq \linearspan\left\{b_T^{\alpha}\big|_T  \colon \alpha \in \mathbb{N}_0^{d+1}, \abs{\alpha} = k \right\} \subseteq \mathcal{P}_{k+d+1}(T).  \label{def:Bk-2d}
\end{align} 
The corresponding global spaces are denoted by 
\begin{align}
	B(\tria_h)&\coloneqq\linearspan\{b_T\colon T\in\tria_h\}\subseteq \mathcal{L}_{d+1}^1(\tria_h),\label{def:global_el_bubblespace}\\
	B_n(\tria_h)&\coloneqq \linearspan\{b_fn_f\colon f\in\facets(\tria_h)\} \subseteq \mathcal{L}_d^1(\tria_h)^d,\label{def:global_normal_face_bubblespace} \\
	\BF(\tria_h)&\coloneqq\linearspan\{b_f^{(i)} n_f\colon  f 
	\in\facets(\tria_h),~i\in\{0,...,d-1\}\} \subseteq \mathcal{L}_{d+1}^1(\tria_h)^d,\\
	B_{k,d}(\tria_h) & \coloneqq \linearspan\{b_T^{\alpha} 
	\colon T \in \tria_h, \alpha \in \mathbb{N}_0^{d+1}, \abs{\alpha} = k \}\subseteq \mathcal{L}_{k+d+1}^1(\tria_h),\label{def:global_Bk-2d}
\end{align}
where~$B_{0,d}(\tria_h) = B(\tria_h)$.  
Furthermore, we denote the spaces restricted to boundary facets by
\begin{align}
	B_n^\partial(\tria_h)&\coloneqq \linearspan\{b_fn_f\colon f\in\facets^\partial(\tria_h)\}\subset B_n(\tria_h), \label{def:bd_face_bubbles}\\ 
	\BFp(\tria_h)&\coloneqq \linearspan\{b_f^{(i)} n_f\colon  f
	\in\facets^\partial(\tria_h),~i\in\{0,...,d-1\}\}\subset \BF(\tria_h).\label{def:global_face_bubblespace_d+1}
\end{align}

\subsection*{Quasi-interpolation operator} 
The construction of Fortin operators is based on the modification of a local quasi-interpolation operator. 
Since trace-preservation properties are crucial in the following, we shall work with Scott--Zhang-type operators~\cite{ScottZhang1990}. 
For convenience we collect their properties. 

\begin{proposition}[{\cite{ScottZhang1990}}]\label{prop:SZinterp}
	Let~$(\tria_h)_h$ be a family of shape-regular triangulations, see \eqref{eq:shape-regularity}, and let~$k\in\N$. 
	Then there is a family of linear projection operators~$I_h\colon  W^{1,1}(\Omega)^d \to X_h=\mathcal{L}_k^1(\tria_h)^d$, i.e., $I_h\big|_{X_h} = \identity$, with the following properties: 
	\begin{enumerate}
		\item\label{itm:Ih-tracedep}
		(traces) Zero (normal) traces are preserved in the sense that 
		\begin{align}\label{eq:Ih-zerotr} 
		I_h(W^{1,1}_0(\Omega)^d) 
		&
		\subset 
			X_{h,0}, \\
			\label{eq:Ih-zeronormaltr}
		I_h(W^{1,1}_n(\Omega))
		 &
		 \subset X_{h, n}; 
		\end{align}
	\item\label{itm:Ih-W11stab} 
	(local $W^{1,1}$-stability) One has
	\begin{align*}
		\norm{I_h v}_{L^1(T)} + h_T \norm{\nabla (I_h v)}_{L^1(T)}
		 &\lesssim \norm{v}_{L^1(\Omega_h^1(T))} + h_T \norm{\nabla v}_{L^1(\Omega_h^1(T))}
	\end{align*}
	for all $v \in W^{1,1}(\Omega)^d$ and for all $T \in \tria_h$, uniformly in $h$; 
	\item\label{itm:Ih-local_approx} (local approximation property) For $p\in[1,\infty]$, $\ell\in\N_0$ and $\max\{\ell,1\}\leq m \leq k+1$ one has
	\begin{align*}
		\norm{\nabla^\ell (v - I_h v)}_{L^p(T)} \lesssim h_T^{m-\ell} \norm{\nabla^m v}_{L^p(\Omega_h^1(T))}
	\end{align*}
	for all $v \in W^{m,p}(\Omega)^d$ and for all $T\in\tria_h$, uniformly in $h$.
\end{enumerate} 
\end{proposition}
\begin{proof}
The Scott--Zhang quasi-interpolation operator introduced in \cite{ScottZhang1990} satisfies all the stated properties: 
the zero-trace preservation in~\ref{itm:Ih-tracedep}~\eqref{eq:Ih-zerotr} and stability in~\ref{itm:Ih-W11stab} are established in~\cite[Thm.~2.1, Thm.~3.1]{ScottZhang1990}. 
 The preservation of zero normal traces in~\ref{itm:Ih-tracedep}~\eqref{eq:Ih-zeronormaltr} is verified in~\cite[Lemma 4.11]{GazcaOrozco2025}. 
 By applying averaged Taylor polynomials and the projection property one can show that~\ref{itm:Ih-W11stab} implies approximation properties in $W^{1,p}$ as in \ref{itm:Ih-local_approx} for any $p\in[1,\infty]$, see \cite[Sec.~4]{Diening2007}.
\end{proof}

In fact, corresponding quasi-interpolation operators are also available for spaces $X_h$ satisfying $ \mathcal{L}^1_k(\tria_h)^d \subsetneq X_h$ for some $k \in \mathbb{N}$. 
The following lemma addresses the general case. 

\begin{lemma}\label{lem:SZinterp}
 Let~$X_h$ be as in~\eqref{def:Xh} with $\mathcal{P}_k(T)^d \subset X(T)$ for some $k\in\N$ and any $T \in \tria_h$, and hence  $\mathcal{L}^1_k(\tria_h)^d\subset   X_h$. 
 Then there exists a family of linear projection operators $I_h\colon W^{1,1}(\Omega)^d\to X_h$ satisfying properties \ref{itm:Ih-tracedep}--\ref{itm:Ih-local_approx} in Proposition~\ref{prop:SZinterp}. 
\end{lemma}
\begin{proof}
	As mentioned in \cite[p.~485]{ScottZhang1990} the construction of the Scott--Zhang quasi-interpolation operator extends to finite element spaces where~$X(T) \supset \mathcal{P}_k(T)^d$.  
	It can be constructed to be a projection operator, i.e., $I_h\big|_{X_h}=\identity$ still holds and it preserves zero (normal) traces as in Proposition~\ref{prop:SZinterp}~\ref{itm:Ih-tracedep}. 
	The remaining properties \ref{itm:Ih-W11stab}--\ref{itm:Ih-local_approx} can be obtained as in the standard case. 
\end{proof}

\begin{remark}[trace-preservation properties]\label{rmk:SZ-discr-traces}
	\hfill 
\begin{enumerate}
	\item  
	\label{itm:SZ-discr-tr}	
 The operator~$I_h$ in Lemma~\ref{lem:SZinterp} preserves discrete traces in the sense that
	\begin{alignat*}{4}
		\tr(I_h(v))&=\tr(v) &&\quad\text{ if } \tr(v)=\tr(v_h) &&&\text{ for some } v_h\in {X_h}, \\
		\tr(I_h (v))\cdot n &= \tr(v)\cdot n &&\quad\text{ if } \tr(v)\cdot n = \tr(v_h)\cdot n \quad&&&\text{ for some } v_h\in {X_h}.
	\end{alignat*}
This is a direct consequence of the projection property on~$X_h$ and the zero-trace-preservation property, see  Proposition~\ref{prop:SZinterp}~\ref{itm:Ih-tracedep}. 
	\item 
	\label{itm:SZ-part-bd} 
	Assume that~$\Gamma\subsetneq\partial\Omega$ is a (sufficiently regular) subset of the boundary which is resolved by the family of triangulations $(\tria_h)_h$, i.e., $\Gamma$ is the union of a subset of boundary facets in $\tria_h$. 
	Then, the operator~$I_h$ as in Lemma~\ref{lem:SZinterp} can be designed such that~$\tr(I_h(v))\big|_{\Gamma}$ only depends on~$\tr(v)\big|_{\Gamma}$, and that zero (normal) traces on~$\Gamma$ are preserved by~$I_h$. 
	This can be achieved by choosing the facets in the construction suitably. More precisely, for any Lagrange node $z \in\Gamma$ one chooses a facet~$f \in \facets(\tria_h)$ with $z \in f \subset \Gamma$. 
	Analogously to~\ref{itm:SZ-discr-tr}, discrete traces on~$\Gamma$ are also preserved. 
\end{enumerate}
\end{remark}

\subsection*{Inf-sup stability and Fortin lemma}
Throughout, we shall work in the following setting. 

\begin{assumption}[FEM setting]
	\label{ass:FEM}
	Let~$(X_h,Q_h)_{h}$ be a family of mixed finite element spaces as defined in~\eqref{def:Xh}, \eqref{def:Qh} on a sequence of shape-regular triangulations~$(\tria_h)_{h}$. 	
	Assume that there exists a~$k\in\N$ such that~$\mathcal{L}_k^1(\tria_h)^d\subseteq X_h$. 
\end{assumption}

We say that the family of pairs $X_{h,0} \times \overline{Q}_h \subset W^{1,\infty}_0(\Omega)^d \times L^{\infty}_0(\Omega)$ as defined in~\eqref{def:Qhbar}, \eqref{def:Xhstar} satisfies the \emph{discrete inf-sup condition} for $p \in (1,\infty)$ if there exists a constant~$ \infsupdiscr =\infsupdiscr(p)>0$ such that
\begin{equation}\label{def:discr-inf-sup-cond}
	\infsupdiscr \norm{q_h}_{L^{p'}(\Omega)} \leq \sup_{v_h\in X_{h,0}\backslash\{0\}} \frac{\skp{\divergence v_h}{q_h}}{\norm{\nabla v_h}_{L^p(\Omega)}} \qquad \text{for all }q_h\in \overline{Q}_h,
\end{equation}
uniformly in $h$, where $p' \in (1,\infty)$ is such that $\tfrac1p + \tfrac{1}{p'}=1$.

Analogously to~\eqref{def:inf-sup-cont}, by the subset relations~\eqref{eq:subspace-relations} this implies the discrete inf-sup condition for $X_{h,n} \times \overline{Q}_h$  and for $X_{h,\sim} \times \overline{Q}_h$ as
\begin{align}\label{est:inf-sup-star}
	\infsupdiscr \norm{q_h}_{L^{p'}(\Omega)}
 \lesssim  
	\sup_{v_h \in X_{h,n} \setminus \{0\}} \frac{\skp{\divergence v_h}{q_h}}{\norm{ v_h}_{W^{1,p}(\Omega)}} \leq 
	\sup_{v_h \in X_{h,\sim} \setminus \{0\}} \frac{\skp{\divergence v_h}{q_h}}{\norm{ v_h}_{W^{1,p}(\Omega)}}.    
\end{align}
By the Fortin lemma~\cite{Fortin1977}, provided that the continuous inf-sup condition holds, the existence of a Fortin operator for~$p=2$ satisfying the divergence-preservation property and global $W^{1,2}$-stability is equivalent to discrete inf-sup stability for~$p=2$; see also~\cite[Lem.~26.9]{Ern2021b}.
The following lemma due to~\cite{Ern2016} represents a generalisation to Banach spaces. 

\begin{lemma}[{\cite[Thm.~1]{Ern2016}}]\label{lem:Fortin_operator}
	Let $(Y,\norm{\cdot}_{Y})$ and $(W, \norm{\cdot}_W)$ be real Banach spaces and let $b \colon Y \times W \to \mathbb{R}$ be a bounded bilinear form which satisfies the inf-sup condition 
	\begin{align*}
		\alpha \leq \inf_{w\in W \setminus \{0\}} \sup_{y \in Y \setminus \{0\}} \frac{b(y,w)}{\norm{w}_W \norm{y}_Y}.  
	\end{align*}
	Let~$Y_N \subset Y$ and $W_N \subset W$, for $N \in \mathbb{N}$,
	be families of finite-dimensional subspaces. 
Then, there is a constant~${ \alpha_*} >0$
	such that the discrete inf-sup condition 
	\begin{align*}
{ \alpha_*} \leq \inf_{w_N\in W_N \setminus \{0\}} \sup_{y_N \in Y_N \setminus \{0\}} \frac{b(y_N,w_N)}{\norm{w_N}_W \norm{y_N}_Y}
\end{align*}
	is satisfied for all~$N\in \mathbb{N}$ if and only if there is an operator $\Pi_N \colon Y \to Y_N$ with the properties: 
	\begin{enumerate}[label = (\roman*)]
		\item \label{itm:F-h-div-pres} $b(\Pi_N(y) - y,w_N)= 0$ \qquad for all  $w_N \in W_N$, for all $y \in Y$;
		\item \label{itm:F-h-stab} there is a constant $c>0$ such that $$\quad\norm{ \Pi_N y  }_{Y} \leq c \norm{y }_{Y}\qquad \text{ for all } y \in Y \; \text{ and any } N \in \mathbb{N}.$$
	\end{enumerate}
	If $Y$ is a Hilbert space, then $\Pi_N$ can be chosen to be linear. 
\end{lemma}
Lemma~\ref{lem:Fortin_operator} has the following consequences: 
 If the discrete inf-sup condition~\eqref{def:discr-inf-sup-cond} holds, then,  by \eqref{def:inf-sup-cont} and  \eqref{est:inf-sup-star} for each~$p \in (1,\infty)$ and for each~$\star  \in \{0,n,\sim\}$ there exists an operator~$\Pi_{h,\star}^{p}\colon W^{1,p}_{\star}(\Omega)^d \to X_{h,\star}$ such that 
\begin{enumerate}[label = (\roman*)]
	\item $\skp{\divergence \Pi^{p}_{h,\star}(v)}{q_h} = \skp{\divergence v}{q_h}$ \qquad for all  $q_h \in \overline{Q}_h$, for all $v\in W^{1,p}_{\star}(\Omega)^d$;
	\item there is a constant $c_{\star,p}>0$ such that 
	$$\norm{ \Pi_{h,\star}^{p} (v) }_{W^{1,p}(\Omega)} \leq c_{\star,p} \norm{ v }_{W^{1,p}(\Omega)} \qquad   \text{ for all } v \in W^{1,p}_{\star}(\Omega)^d.$$
\end{enumerate} 
However, the abstract Fortin lemma does not provide a single operator $\Pi_h \colon W^{1,1}(\Omega)^d\to X_h$ which is stable in~$W^{1,p}(\Omega)^d$ for all~$p \in (1,\infty)$ and preserves homogeneous traces in the sense that~$\Pi_h(W^{1,1}_\star(\Omega)^d)\subseteq X_{h,\star} \subset W^{1,\infty}_\star(\Omega)^d$
for each~$\star\in\{0,n,\sim\}$. 
Furthermore, in general the operators are neither projections nor local, and the abstract Fortin lemma guarantees linearity only for $p = 2$.  

For these reasons, for a range of finite element spaces we construct a single operator satisfying all these properties.


\section{Properties of Fortin operators}
\label{sec:main_results}

Let us collect the desired properties of Fortin operators, which we verify in the construction. 
Let Assumption~\ref{ass:FEM} be satisfied with~$k \in \mathbb{N}$.  
For specific finite element spaces we construct an operator~$\Pi_h\colon W^{1,1}(\Omega)^d\to X_h$ with the following properties:

\begin{enumerate}[label = (P\arabic*)]\setcounter{enumi}{-1}
	\item \label{prop:lin-proj} 
	(linear projection) 
	The operator~$\Pi_h$ is linear and a projection onto~$X_h$ in the sense that~$\Pi_h\big|_{X_h}=\identity$;
	\item \label{prop:div-pres} 
	(divergence preservation)
	For all $v\in W^{1,1}(\Omega)^d$ we have that
	\begin{align}\label{eq:prop-div_pres}
		\skp{\divergence (\Pi_h v)}{q_h} 
		= \skp{\divergence v}{q_h} \qquad \text{for all }q_h\in \overline{Q}_h;
	\end{align}
	\item \label{prop:trace-pres} 
	(trace preservation) 
	$\Pi_h$ preserves zero traces in the following sense   
	\begin{enumerate}[label=(\roman*)]
		\item \label{itm:global_zerotrace} $\Pi_h(W^{1,1}_0(\Omega)^d)\subseteq X_{h,0}$,
		\item \label{itm:global_zeronormaltrace} $\Pi_h(W^{1,1}_n(\Omega))\,\;
		\subseteq X_{h,n}$,
		\item \label{itm:global_meantrace} 
		$\Pi_h(W^{1,1}_\sim(\Omega))\,\;
		\subseteq X_{h,\sim}$;
	\end{enumerate}

	\item \label{prop:global-stab-approx} (global stability and approximation) For any~$p\in[1,\infty]$ we have
	\begin{enumerate}[label=(\roman*)]
		\item \label{itm:prop_global_H1semi-stab}$\displaystyle \norm{\nabla (\Pi_h v)}_{L^p(\Omega)} \lesssim \norm{\nabla v}_{L^p(\Omega)}$ 
		\quad for any $v\in W^{1,p}(\Omega)^d$;
		\item 
		\label{itm:prop_global_Sob-stab}
		$\displaystyle \norm{\Pi_h v}_{W^{1,p}(\Omega)}\;\;\, \lesssim \norm{v}_{W^{1,p}(\Omega)}$ 
		\quad  for any $v\in W^{1,p}(\Omega)^d$;
		\item \label{itm:prop_global_approx} 
		 for $\ell \in \{0,1\}$ and $m \in \mathbb{N}$ with $ \ell\leq m\leq k+1$ we have that 
		\begin{align*}
			\norm{\nabla^\ell (v-\Pi_h v)}_{L^p(\Omega)} \lesssim h^{m-\ell} \norm{ \nabla^m v}_{L^p(\Omega)} \quad \text{for any } v\in W^{m,p}(\Omega)^d,
		\end{align*}
	\end{enumerate}
	independently of~$h$; 
	\item \label{prop:local-stab-approx}(local stability and approximation) There is an~$s\in\N$ such that for~$p\in[1,\infty]$ we have 
	\begin{align}\label{eq:prop_local-stab}
		\norm{\Pi_h v}_{L^p(T)} + h_T \norm{\nabla (\Pi_hv)}_{L^p(T)}\lesssim \norm{v}_{L^p(\Omega^s_h(T))} +  h_T \norm{\nabla v}_{L^p(\Omega_h^s(T))}
	\end{align}
	for all~$v\in W^{1,p}(\Omega)^d$, for all~$T\in \tria_h$, uniformly in $h$. 
	 Further, for any~$\ell\in\N_0$ and $m \in \mathbb{N}$ with $\ell\leq m\leq k+1$ we have
	\begin{align}\label{eq:prop_local-approx}
		\norm{\nabla^\ell (v-\Pi_h v)}_{L^p(T)} \lesssim h_T^{m-\ell} \norm{\nabla^m v}_{L^p(\Omega_h^s(T))} \qquad \text{for any } v\in W^{m,p}(\Omega)^d,
	\end{align}
	for all~$T\in \tria_h$, uniformly in $h$. 
The hidden constants in~\eqref{eq:prop_local-stab}, \eqref{eq:prop_local-approx} depend only on the shape-regularity constant of $(\tria_h)_{h}$, $d$, $s$, $p$,  $\dim X(T)$ and~$\dim Q(T)$. 
\end{enumerate}

\begin{table}[h] 
	\small
	\begin{TAB}(r)[3pt]{|l|l|l|}{|c|cc|ccc|ccc|c|cccc|cccc|cc|c|}
		finite	element	& properties & results and conditions \\
		$P_d-P_0$
		& \ref{prop:lin-proj}--\ref{prop:local-stab-approx} with $s = 1$
		& 
		\\
	& \quad for $k = d$
	& Rmk.~\ref{rmk:PdP0}
	 \\
		Bernardi--
		& \ref{prop:lin-proj}--\ref{prop:local-stab-approx} with $s = 1$ 
		& \\
	Raugel	& \quad for $k=1$, for $d \in \{2,3\}$  &  Lem.~\ref{lem:BR_Fortin-v} \\
		& \quad for $k = 2$ for $d = 3$   &  Rmk.~\ref{rmk:BR-2ndorder}\\
		conforming   
		& \ref{prop:lin-proj}--\ref{prop:local-stab-approx} with $s=1$
		& 
	\\
	Crouzeix--	& \quad  for $k \geq 2$
		& 
		Prop.~\ref{prop:Fortin_CR}: $k\ge d$ 	 \\
	Raviart	& & Rmk.~\ref{rmk:confCRd3k2}: $d=3$, $k=2$, enrichment with $B_n(\tria_h)$ \\
		Guzmán--Neilan
		& \ref{prop:lin-proj}--\ref{prop:local-stab-approx}
		& Section~\ref{subsec:Guzman-Neilan}, \cite{Guzman2014a,Guzman2014b}
		\\
		modified 
		& \ref{prop:lin-proj}--\ref{prop:local-stab-approx} with $s=1$
		& 
	 \\
	MINI	& \quad for $k=1$
		&
			Lem.~\ref{lem:Fortin_MINI}: enrichment with $B_n(\tria_h)$
	 \\
	 && Rmk.~\ref{rmk:MINI-alternative}:  enrichment with $\BFp(\tria_h)$\\
	 		& \quad for $k \geq 1$ & 
	 		Rmk.~\ref{rmk:MINI-highorder}: enrichment with $B_n(\tria_h)$ if $k < d$  \\ 
		Taylor--
		& \ref{prop:lin-proj}--\ref{prop:local-stab-approx}  with $s=2$
		& 
		 \\
	Hood	& \quad  for $k \geq 2 $
		& 
		Prop.~\ref{prop:Fortin_TH}: $k\ge d$
	\\
	& & 
		Rmk.~\ref{rmk:THd3k2}: $d=3$, $k=2$, enrichment   \\
		&& \qquad \qquad \quad  with $B_n(\tria_h)$ or $\BFp(\tria_h)$
		\\
		all inf-sup stable pairs
		& \ref{prop:div-pres}--
			\ref{prop:global-stab-approx} for fixed $p$ 
		& 
		Prop.~\ref{prop:Fortin-global}: if $B_n(\tria_h)\subseteq X_h$
		\\
		&  \qquad for $\ell = 1$ & \\
		all inf-sup stable pairs
		& \ref{prop:lin-proj}--\ref{prop:global-stab-approx} for $p = 2$ and $\ell = 1$ 
		&  Rmk.~\ref{ex:globalFortin}~\ref{itm:example-p2}: for $p = 2$ 
	\end{TAB}
	\caption{Overview of properties of Fortin operators for various mixed finite element spaces satisfying\\
		 Assumption~\ref{ass:FEM} with $k \in \mathbb{N}$ in dimension $d \in \{2,3\}$.}
	\label{tbl:fortin-summary}
\end{table}

\begin{remark}\label{rmk:properties_Fortin}
	\begin{enumerate}
		\item 
		\label{rmk:stab_approx} 
	The local properties in~\ref{prop:local-stab-approx} imply the global ones in~\ref{prop:global-stab-approx}. 
	The proof proceeds by summing the $p$th  powers of the estimates over~$T\in \tria_h$ for~$p \in [1,\infty)$, and by taking the maximum   over~$T\in \tria_h$ for~$p = \infty$ using the shape-regularity of~$(\tria_h)_h$. 	
	If the operators $\Pi_h$ are local as in~\ref{prop:local-stab-approx}, then the constants in~\ref{prop:global-stab-approx}  depend only on the  shape-regularity constant of $(\tria_h)_{h}$, $d$, $s$, $p$,  $\dim X(T)$ and $\dim Q(T)$. Note that $s\in\N$  depends only on the specific finite element pair.
	
	Also, local stability~\eqref{eq:prop_local-stab} follows from the local approximation property~\eqref{eq:prop_local-approx}. 
	Hence, to verify \ref{prop:global-stab-approx} and \ref{prop:local-stab-approx} it suffices to show \eqref{eq:prop_local-approx}. 
	\item 
		\label{itm:rmk_discr-trace} 
	Analogously to Remark~\ref{rmk:SZ-discr-traces}~\ref{itm:SZ-discr-tr}, the linear operator~$\Pi_h$ preserves discrete (normal) traces	of functions in~$X_h$ as a consequence of \ref{prop:trace-pres}~\ref{itm:global_zerotrace} and \ref{itm:global_zeronormaltrace} and the projection property \ref{prop:lin-proj}. 
		\item 
		\label{itm:rmk_part-bd}
		 Assume that~$\Gamma\subsetneq\partial\Omega$ is a sufficiently regular subset of the boundary of positive $(d-1)$-dimensional measure, which is resolved by $(\tria_h)_h$, meaning that it is the union of a subset of the boundary facets of $\tria_h$. 
		 For all finite element pairs we consider in this work the operator~$\Pi_h$ can be constructed to preserve zero (normal) traces  on~$\Gamma$ without imposing zero traces on $\partial \Omega \setminus \Gamma$.  
		We refer to this stronger property as~\ref{prop:trace-pres}$_{\Gamma}$. 
		For mixed boundary conditions, see also \cite[Ch.~6]{Bernardi2024}. 
		\item 
		The local Fortin operators constructed in Section~\ref{sec:localFortin} satisfy property~\ref{prop:div-pres} also for pressure functions~$q_h\in Q_h = \overline{Q}_h \oplus \linearspan\{1\}$. 
		This directly implies~\ref{prop:trace-pres}~\ref{itm:global_meantrace}. 
		\item 
		The Fortin operators for the~$P_d-P_0$,  lowest-order Bernardi--Raugel ~\cite{BernardiRaugel1985},  conforming Crouzeix--Raviart \cite{CrouzeixRaviart1973} and Guzmán--Neilan element~\cite{Guzman2014a, Guzman2014b} satisfy all properties~\ref{prop:lin-proj}--\ref{prop:local-stab-approx} by construction. 
		For the sake of completeness we review them in Section~\ref{subsec:discont} below. 
	\end{enumerate}
\end{remark}


\section{Applications}\label{sec:applications}
Let us collect some consequences of the availability of a Fortin operator satisfying properties~\ref{prop:lin-proj}--\ref{prop:local-stab-approx}. 
In particular, preservation of zero traces, zero  normal traces, and zero mean of the normal traces, as well as locality, are important in the following applications.

\subsection{Quasi-optimal error estimates for inhomogeneous Dirichlet boundary conditions}
For general mixed finite element spaces  with a Fortin operator $\Pi_h$ satisfying properties~\ref{prop:lin-proj}--\ref{prop:global-stab-approx} quasi-optimal error estimates for the Stokes equations with inhomogeneous Dirichlet boundary conditions have been established in~\cite[Thm.~2.13]{Eickmann2025a}. 
More precisely, let $(u,\pi) \in W^{1,2}(\Omega)^d \times L^2_0(\Omega)$ be the solution to the Stokes problem with $\tr(u) = \tr(g)$ on $\partial \Omega$ for some $g\in W^{1,2}_{\sim}(\Omega)$, and let $(u_h,\pi_h) \in X_h \times Q_h$ be the solution to the corresponding discrete Stokes problem with $\tr(u_h) = \tr(\Pi_h g)$ on $\partial \Omega$. 
Then, one has 
\begin{equation}
	\begin{aligned}
	\norm{\nabla u - \nabla u_h}_{L^2(\Omega)}
	&+ 
	\nu^{-1} \norm{  \pi- \pi_h}_{L^2(\Omega)} \\
	 &\lesssim \inf_{v_h \in X_h} \norm{\nabla u - \nabla v_h}_{L^2(\Omega)} + \nu^{-1} \inf_{q_h \in Q_h}  	\norm{\pi - q_h}_{L^2(\Omega)},
\end{aligned}
\end{equation}
uniformly in $h$. 
This holds under the assumption that~$\Pi_h$ preserves traces of functions in~$X_h$, see~\cite[Assump.~2.10 (iiia)]{Eickmann2025a}.
The \textit{discrete trace-preservation property} of~$\Pi_h$  allows the treatment of the trace condition as an additional constraint, and hence yields quasi-optimality without extra data approximation terms.

The Scott--Vogelius element for~$d = 2$ is defined by~$X_h = \mathcal{L}^1_k(\tria_h)^2$ and~$Q_h = \divergence X_{h,\sim}$, and is inf-sup stable for~$k \geq 4$. 
On triangulations without boundary singular vertices, a quasi-optimal error estimate for the velocity holds, in the sense that 
 \begin{align}
	\norm{\nabla u - \nabla u_h}_{L^2(\Omega)}
	\lesssim \inf_{v_h \in X_h} \norm{\nabla u - \nabla v_h}_{L^2(\Omega)},
 \end{align}
 see~\cite[Thm.~3.4]{Eickmann2025a}. 
 Such an estimate is referred to as a \emph{pressure-robust} error estimate. 
The trace-preservation property~\ref{prop:trace-pres}~\ref{itm:global_meantrace} is key for exact divergence constraints.

\subsection{Non-Newtonian fluid flow with Dirichlet boundary conditions}
\textit{Locality} of the Fortin operator as in~\ref{prop:local-stab-approx} is instrumental for the analysis of non-Newtonian fluid flows. 
In the seminal work~\cite{BelenkiBerselliDieningEtAl2012} a priori error estimates for general finite element discretisations of the~$p$-Stokes system and generalisations subject to homogeneous Dirichlet conditions were established. 
Locality of the Fortin operator (see Assump.~2.9~(c) in~\cite{BelenkiBerselliDieningEtAl2012}) is important to handle the nonlinear term in a so-called quasi-norm. 
Those results were subsequently extended to more general equations, but under the same assumptions on the Fortin operator, see, e.g.,~\cite{BerselliKaltenbach2025}. Furthermore, locality was exploited in the construction of a discrete divergence-free Lipschitz truncation in~\cite{DieningKreuzerSueli2013}, which is used therein to show plain convergence results for generalised Navier--Stokes equations. 

In~\cite{Jessberger2024} a priori error estimates for non-Newtonian fluids with homogeneous Dirichlet boundary conditions were extended to inhomogeneous Dirichlet boundary conditions. 
This relies on a lifting operator based on adapted assumptions on the Fortin operator, see~\cite[Assump.~2.10]{Jessberger2024}.
It requires the inhomogeneous version of the {divergence preservation}\footnote{A minor typo in~\cite[Assump.~2.9]{BelenkiBerselliDieningEtAl2012} was carried over to~\cite{Jessberger2024},  leading to the claim that the assumptions hold for several elements. 	
To the best of our knowledge, their validity is established only in the present work.} in~\ref{prop:div-pres} to hold rather than the homogeneous version for~$v \in W^{1,1}_0(\Omega)^d$. 
One can show that this property is satisfied for the Bernardi--Raugel element, the $P_d-P_0$ element, the conforming Crouzeix--Raviart element and the Guzmán--Neilan element with their respective standard Fortin operators, see Section~\ref{subsec:discont} below. 
However, for other finite elements this is nontrivial. 
In Section~\ref{sec:localFortin} we verify these properties  for the (enriched) MINI and Taylor--Hood elements.

\subsection{Domain approximation and moving domains}
\emph{Locality} of the Fortin operator, as in ~\ref{prop:local-stab-approx}, can be used  to obtain discrete inf-sup stability on sequences of polygonal domains, which is uniform with respect to the domain sequence. 
This is useful when approximating curved domains and in the context of moving domains approximated with  unfitted meshes. 
In dimension~$d = 3$ the polygonal domains occurring in such applications need not be Lipschitz. 
Under reasonable assumptions, however, they may still be so-called John domains~\cite{John1961}, see also \cite{AcostaDuranMuschietti2006,DieningRuzickaSchumacher2010}. 

By the local stability property~\ref{prop:local-stab-approx}, the Fortin operators are uniformly stable, with constants depending only on the shape regularity of the triangulation. 
Consequently, uniform discrete inf-sup stability with respect to both the domain sequence and the discretisation follows by the standard arguments, cf.~\cite[Prop.~5.4.2]{Boffi2013}.  
A different approach based on alternative assumptions on the discrete domains was used  in~\cite{KirchhartGrossReusken2016,GuzmanOlshanskii2018} to establish uniform discrete inf-sup stability. 

We consider a sequence of John domains~$(\Omega_N)_{N \in \mathbb{N}}$, as introduced by~\cite{John1961}, see also~\cite[Def.~2.1]{AcostaDuranMuschietti2006} or~\cite[Def.~3.1, Def.~3.5]{DieningRuzickaSchumacher2010}, with polyhedral boundary and without slits. 
For such domains Sobolev spaces and traces can be defined as above. 
Let~$(\tria_{N})_{N \in \mathbb{N}}$ be a sequence of triangulations where~$\tria_N$ triangulates~$\Omega_N$ for each $N \in \mathbb{N}$  with maximal mesh size~$h_N>0$.  
We consider finite-dimensional mixed pairs~$X_{N,0} \times \overline Q_N \subset W^{1,\infty}_0(\Omega_N)^d \times L^{\infty}_0(\Omega_N)$ on the respective triangulation~$\tria_N$, and aim for a uniform discrete inf-sup constant. 

More precisely, for every~$p \in (1,\infty)$ we seek a constant~$\infsupdiscr = \infsupdiscr(p)>0$, independent of $N$ and~$h_N$,  such that 
\begin{align}\label{est:discr-inf-sup-Omega-n} 
	\infsupdiscr 
	 \leq 
	\inf_{q_N \in \overline Q_N \setminus \{0\}} \sup_{v_N \in X_{N,0} \setminus \{0\}}	 \frac{\skp{\divergence v_N}{q_N}_{\Omega_N}}{\norm{q_N}_{L^{p'}(\Omega_N)} \norm{\nabla v_N}_{L^p(\Omega_N)}} \qquad \text{ for any } N \in \mathbb{N}. 
\end{align} 
The argument relies on two ingredients. First, we need the uniform inf-sup stability of the pair~$W^{1,p}_0(\Omega_N)^d \times {L^{p'}_0(\Omega_N)}$, meaning that 
\begin{align}
	\label{est:inf-sup-Omega-n} 
	\tilde{\beta} \leq  
	\inf_{q \in L^{p'}_0(\Omega_N) \setminus \{0\}} \sup_{v \in W^{1,p}_0(\Omega_N)^d \setminus \{0\}} \frac{\skp{\divergence v}{q}_{\Omega_N}}{\norm{q}_{L^{p'}(\Omega_N)} \norm{\nabla v}_{L^p(\Omega_N)}} \qquad \text{ for any } N \in \mathbb{N}, 
\end{align} 
holds with constant~$\tilde \beta = \tilde \beta(p)>0 $ independent of $N$. And second, we use the existence of a local Fortin operator~$\Pi_N\colon W^{1,p}(\Omega_N)^d\to X_N$ satisfying, in particular,~\ref{prop:local-stab-approx}. 

	An estimate of the form~\eqref{est:inf-sup-Omega-n} follows from~\cite[Thm.~4.1]{AcostaDuranMuschietti2006} if the domains~$\Omega_N$ are John domains with John constants uniformly bounded in~$N$, see also~\cite[Thm.~5.2]{DieningRuzickaSchumacher2010} for estimates in weighted spaces using a decomposition technique. 
	For the special case of Lipschitz domains~$(\Omega_N)_{N \in \mathbb{N}}$ approximating a Lipschitz domain~$\Omega$ sufficiently well, and $p = 2$, \eqref{est:inf-sup-Omega-n} was also proven in~\cite[Thm.~4.4]{BernardiCostabelDaugeEtAl2016}.  
	For~$p = 2$ and certain unions of star-shaped domains, the dependence of~$\tilde \beta$ in  
	\eqref{est:inf-sup-Omega-n} on geometric parameters was investigated in~\cite[Cor.~28, Thm.~35]{GuzmanSalgado2021}. 

For domains satisfying~\eqref{est:inf-sup-Omega-n}  the standard Fortin argument yields discrete inf-sup stability~\eqref{est:discr-inf-sup-Omega-n} for~$p \in (1,\infty)$, and hence in particular for non-Newtonian fluids. 
Notably, the domains need not be Lipschitz.    
The following result holds for~$p \in (1,\infty)$, whereas  ~\cite[Thm.~1]{GuzmanOlshanskii2018} is presented for~$p = 2$. 
While the arguments are essentially standard, we include the details for the sake of completeness, as we are not aware of suitable references. 

\begin{theorem}[inf-sup stability for domain approximation]\label{thm:domain approximation}
	
	Let~$K \subset \setR^d$ be a bounded simply connected Lipschitz domain with polyhedral boundary.
	Let~$(\mathcal{S}_N)_{N \in \mathbb{N}}$ be a sequence of shape-regular triangulations of~$K$. 
	For each~$N \in \mathbb{N}$ let~$\tria_N \subset \mathcal{S}_N$ and define the bounded domain~$\Omega_N \coloneqq \interior(\cup_{T \in \tria_N} T)$, which has polyhedral boundary and no slits, and assume that $\Omega_N$ is simply connected. 
	Assume that for~$p \in (1,\infty)$ there is a constant~$\tilde \beta $ such that~\eqref{est:inf-sup-Omega-n} holds uniformly in $N \in \mathbb{N}$.
	Let~$(X_N,Q_N)$ be as defined in~\eqref{def:Xh} and~\eqref{def:Qh} with respect to~$\tria_N$ and define~$X_{N,0} \coloneqq X_N \cap W^{1,\infty}_0(\Omega_{N})^d$.  
	Assume that there is a sequence of local Fortin operators~$\Pi_N \colon W^{1,1}(\Omega_N)^d \to X_N$ that satisfy~\ref{prop:div-pres}, \ref{prop:trace-pres}~\ref{itm:global_zerotrace} and~\ref{prop:local-stab-approx} uniformly in $N \in \mathbb{N}$. 
	
	Then, there is a constant~$\infsupdiscr>0$ depending only on~$p, \tilde \beta$ and on the shape-regularity constant of~$\tria_N$, $d$, $s$,  $\dim X(T)$ and~$\dim Q(T)$, 
	 such that the discrete inf-sup condition~\eqref{est:discr-inf-sup-Omega-n} holds uniformly in $N$.       
\end{theorem} 

\begin{proof} Note that for~$v \in W^{1,p}_0(\Omega_N)^d$, by trace preservation~in~\ref{prop:trace-pres}~\ref{itm:global_zerotrace} one has  $\Pi_N v \in X_{N,0}$. 
	By Poincaré's inequality, it follows that~$ \norm{\nabla \Pi_N v }_{L^p(\Omega_N)} = 0$ if and only if~$ \Pi_N v \equiv 0$. 
For such~$v$ it follows with \ref{prop:div-pres} that  
\begin{align*}
 	\skp{\divergence v}{ q_N} = \skp{\divergence \Pi_N v}{ q_N} = 0 \qquad \text{ for all } q_N \in \overline Q_N. 
\end{align*}
For this reason, using the inf-sup stability in~\eqref{est:inf-sup-Omega-n} for any~$q_N \in \overline Q_N \subset L^{p'}_0(\Omega_N)$, and the stability property~\ref{prop:global-stab-approx}~\ref{itm:prop_global_H1semi-stab} we obtain
\begin{align*}
\tilde \beta 	\norm{q_N}_{L^{p'}(\Omega_N)}  
&\leq 
 \sup_{v \in W^{1,p}_0(\Omega_N)^d \setminus \{0\}} \frac{\skp{\divergence v}{q_N}_{\Omega_N}}{ \norm{\nabla v}_{L^p(\Omega_N)}} \\
 &= 
 \sup_{v \in W^{1,p}_0(\Omega_N)^d 
 \colon \nabla \Pi_N v \not\equiv 0} \frac{\skp{\divergence \Pi_N v}{q_N}_{\Omega_N}}{ \norm{\nabla \Pi_N v}_{L^p(\Omega_N)}}  \frac{ \norm{\nabla \Pi_N v}_{L^p(\Omega_N)}}{ \norm{\nabla v}_{L^p(\Omega_N)}}
\\
& \leq c  \sup_{v_N \in X_{N,0}  \setminus \{0\}} 
 \frac{\skp{\divergence v_N}{q_N}_{\Omega_N}}{ \norm{\nabla  v_N}_{L^p(\Omega_N)}}. 
\end{align*}
Furthermore, due to the locality of~$\Pi_N$ as in~\ref{prop:local-stab-approx} the constant~$c>0$ depends only on the shape regularity of~$(\tria_N)_{N \in \mathbb{N}}$, $p$, $d$, $s$,  $\dim X(T)$ and $\dim Q(T)$, 
 see Remark~\ref{rmk:properties_Fortin}~\ref{rmk:stab_approx}.  
 	Then the claim follows with~$\infsupdiscr \coloneqq c^{-1} \tilde \beta$.  
\end{proof}

The background triangulation of $K$ is not needed in the proof if one works with a sequence of simply connected domains $(\Omega_N)_{N \in \mathbb{N}}$ with polyhedral boundary without slits and shape-regular triangulations $(\tria_N)_{N \in \mathbb{N}}$. 
However, in applications of unfitted meshes, often a domain $K$ is available and may, for example, be chosen as a hyperrectangle.

\begin{remark}\label{rmk:Bogovskii} 	
The inf-sup stability~\eqref{est:inf-sup-Omega-n} for~$p \in (1,\infty)$ uniformly in~$N \in \mathbb{N}$ follows from  the existence of bounded linear \Bogovskii{} operators~$\mathcal{B}_N \colon L^{p}_0(\Omega_N)\to W^{1,p}_0(\Omega_N)^d$ satisfying 
	\begin{alignat*}{3}
		\divergence \mathcal{B}_N q &= q \qquad &&\text{ for any } q \in L^{p}_0(\Omega_N),\\
		\norm{\nabla \mathcal{B}_N q}_{L^p(\Omega_N)} &\leq c \norm{q}_{L^p(\Omega_N)} \qquad &&\text{ for any }  q \in L^{p}_0(\Omega_N),
	\end{alignat*}  
	uniformly in $N \in \mathbb{N}$. 
	Thanks to~\cite[Thm.~4.1]{AcostaDuranMuschietti2006} and \cite[Thm.~5.2]{DieningRuzickaSchumacher2010} existence of a \Bogovskii{} operator (independent of $p$) is guaranteed with stability constant~$c$  depending only on the exponent~$p\in (1,\infty)$ and constants~$\sigma_1, \sigma_2>0$, provided that~$(\Omega_N)_{N \in \mathbb{N}}$ are John domains with uniform constants appearing in the notion of John domains. For example, one can use~$\sigma_1, \sigma_2>0$ in the definition of the \emph{emanating chain condition} in~\cite[Def.~3.11]{DieningRuzickaSchumacher2010}, see also Theorem~3.8 therein.
The definition of a sequence of John domains with uniform John constants excludes for example bounded domains~$(\Omega_N)_{N \in \mathbb{N}}$ which degenerate locally for a single $N \in \mathbb{N}$, or as~$N \to \infty$. 
Furthermore, it has been shown in~\cite[Cor.~1.1]{JiangKauranenKoskela2014} that in dimension $d = 2$ for a bounded, simply connected domain $\Omega_N$ the inf-sup stability~\eqref{est:inf-sup-Omega-n} holds if and only if $\Omega_N$ is a John domain. 
This highlights the relevance of the notion of John domains in this context. 

In~\cite{GuzmanOlshanskii2018}~the authors consider a sequence of discrete domains~$(\Omega_N)_{N \in \mathbb{N}}$, with~$\Omega_N \subset \Omega$ for a fixed bounded Lipschitz domain~$\Omega$.     
Their Assumptions 1 and 2 exclude slit domains~$\Omega_N$ and are satisfied for sufficiently fine approximations of a Lipschitz domain $\Omega$.  
We expect that, in their setting, the sequence~$(\Omega_N)_{N \in \mathbb{N}}$ can be shown to consist of John domains with uniform constants, provided that the mesh size is sufficiently small. 
In this case, Theorem~\ref{thm:domain approximation} would be an extension of Thm.~1 in \cite{GuzmanOlshanskii2018}. 
\end{remark}


\section{Local Fortin operators}\label{sec:localFortin}

In this section we construct local Fortin operators for several finite elements. 
After summarising the general framework in Section~\ref{sec:framework-constr} we first consider finite element pairs with discontinuous pressure functions in Section~\ref{subsec:discont}. 
We review the construction for the Bernardi--Raugel element~\cite{BernardiRaugel1985} for~$d\in\{2,3\}$, the $P_d-P_0$ element~\cite[Sec.~8.4.3]{Boffi2013}, the conforming Crouzeix--Raviart element for~$d\in\{2,3\}$ ~\cite{CrouzeixRaviart1973} and its higher-order generalisation for $k \geq d$~\cite{Mansfield1982}, and the Guzmán--Neilan element~\cite{Guzman2014a,Guzman2014b}. 
For those elements, the Fortin operators constructed in the literature for the homogeneous case do not need any modification to satisfy~\ref{prop:lin-proj}--\ref{prop:local-stab-approx}.  
We present their construction in a unifying framework and we prove the extended list of properties.

Then, in Section~\ref{subsec:cont} we consider finite element pairs with continuous pressure functions.  
 This includes a MINI type element~\cite{ArnoldBrezziFortin1984} for which we modify the velocity space in order to construct a local trace-preserving Fortin operator. 
 For the generalisation~\cite[Sec.~8.8.2]{Boffi2013} of the Taylor--Hood element~\cite{Taylor1973} for $k\geq d$ we present a local construction of a Fortin operator satisfying~\ref{prop:lin-proj}--\ref{prop:local-stab-approx}.

\subsection{General construction principle}\label{sec:framework-constr}
For the construction of a local Fortin operator it is customary to  start from a (quasi)-interpolation operator and correct it to ensure the divergence preservation. 
We refer to~\cite{ArnoldBrezziFortin1984} for an early contribution in this spirit, and to~\cite[Prop.~5.4.4]{Boffi2013} for the general principle.  

Let us review the strategy for the homogeneous case of a linear Fortin operator~$\Pi_{h,0}\colon W^{1,1}_0(\Omega)^d \to X_{h,0}$ preserving the divergence of functions $v\in W_0^{1,1}(\Omega)^d$ in the dual of the pressure space. 
The operator is defined by 
\begin{align}\label{def:Fortin-hom}
\Pi_{h,0} v\coloneqq I_h v + C_h(v-I_hv) \qquad \text{ for  } v\in W^{1,1}_0(\Omega)^d,
\end{align}
where~$I_h \colon W^{1,1}_0(\Omega)^d \to X_{h,0}$ is a suitable linear quasi-interpolation operator such as the Scott--Zhang operator~\cite{ScottZhang1990}, and~$C_h \colon W^{1,1}_0(\Omega)^d \to X_{h,0}$ is a linear operator constructed to satisfy
\begin{align}\label{eq:Ch_div-pres}
	\int_{\Omega} \divergence (v-C_hv)~q_h \dx = 0 \qquad \text{for all }q_h\in \overline{Q}_h. 
\end{align}
This ensures the divergence-preservation property of $\Pi_{h,0}$.  
One may equivalently formulate~\eqref{def:Fortin-hom} as 
\begin{align}\label{eq:constr_hom}
	(I-\Pi_{h,0}) = (I-C_h)(I-I_h).
\end{align}
If both operators~$I_h,C_h$ are local, then~$\Pi_{h,0}$ is local as well, and approximation properties of $\Pi_{h,0}$ follow from approximation properties of~$I_h$ and stability properties of~$C_h$. 
Furthermore, $\Pi_{h,0}$ is a linear projection if $I_h$ is. 
The operator $\Pi_{h,0}$ can be extended to $W^{1,1}(\Omega)^d$ with preservation of zero (normal) traces as in~\ref{prop:trace-pres}~\ref{itm:global_zerotrace},  \ref{itm:global_zeronormaltrace}, if~$I_h, C_h$ preserve those zero traces.

We extend this design principle to the case of~$\Pi_h\colon W^{1,1}(\Omega)^d\to X_h$ satisfying the divergence preservation~\ref{prop:div-pres} for all~$v\in W^{1,1}(\Omega)^d$ and all trace-preservation properties~\ref{prop:trace-pres}~\ref{itm:global_zerotrace}--\ref{itm:global_meantrace}. 
Starting from the Scott--Zhang operator, see Proposition~\ref{prop:SZinterp} and Lemma~\ref{lem:SZinterp}, we correct it with respect to trace preservation and divergence preservation for functions with nonzero trace. 
Specifically, we design~$\Pi_h\colon W^{1,1}(\Omega)^d\to X_h$ via  
\begin{align}\label{eq:constr_inhom}
	(I-\Pi_h) = (I-C_h)(I-I_h) = (I-C_h^1)(I-C_h^0)(I-I_h),
\end{align}
with two correction operators $C_h^j\colon W^{1,1}(\Omega)^d\to X_h$ for $j \in \{0,1\}$. 
We now collect sufficient conditions on the operators~$C_h^j$ to ensure that the Fortin operator $\Pi_h$ in~\eqref{eq:constr_inhom} satisfies properties~\ref{prop:lin-proj}--\ref{prop:local-stab-approx}. 

\begin{proposition}\label{prop:Chj_operators}
Let Assumption~\ref{ass:FEM} be satisfied with~$k\in\N$ and consider the quasi-interpolation operator~$I_h\colon W^{1,1}(\Omega)^d\to X_h$ as in Lemma~\ref{lem:SZinterp}. 
	Furthermore, let $C_h^0 \colon  W^{1,1}(\Omega)^d\to X_h$ and  $C_h^1 \colon  W^{1,1}(\Omega)^d\to X_{h,0}$ be bounded linear operators such that the following properties hold: 
	\begin{enumerate}
		\item\label{itm:lemma_div-pres} 
		The operator~$C_h\coloneqq C_h^0+C_h^1(I-C_h^0)$ satisfies 
		\begin{align*}
			\int_{\Omega} \divergence(v-C_hv) q_h \dx = 0 \qquad \text{for any } q_h\in\overline{Q}_h\text{ and any }v\in W^{1,1}(\Omega)^d;
		\end{align*}
		\item\label{itm:lemma_stab_approx} 
		For~$j \in \{0,1\}$, the operator $C_h^j$ is locally $W^{1,1}$-stable, i.e., there exists $s(j)\in\N_0$ such that
		\begin{align}\label{eq:Chj-stab}
			\dashint_T |C_h^j v|\dx  \lesssim \dashint_{\Omega_h^{s(j)}(T)} |v| \dx + h_T \dashint_{\Omega_h^{s(j)}(T)} |\nabla v| \dx 
			\qquad \text{for any } v\in W^{1,1}(\Omega)^d,
		\end{align}
	for any~$T \in \tria_h$, uniformly in~$h$; 
		
		\item\label{itm:lemma_trace-int} 
		The operator~$C_h^0$ satisfies~$C_h^0(W^{1,1}_0(\Omega)^d)
			\subset X_{h,0}$, \, $C_h^0(W^{1,1}_n(\Omega)) \subset X_{h,n} $ and 
			\begin{align}\label{itm:Ch0-trace-int}
			\quad 	\int_{f} \tr(v-C_h^0 v)\cdot n \dsig = 0
			\qquad \text{ for any } v\in W^{1,1}(\Omega)^d~\text{ and any } f\in\facets^\partial(\tria_h).
			\end{align} 
	\end{enumerate}
	Then~$\Pi_h\colon W^{1,1}(\Omega)^d\to X_h$, as defined in \eqref{eq:constr_inhom}, satisfies properties \ref{prop:lin-proj}--\ref{prop:local-stab-approx} with $s=s(0)+s(1)+1$. 
	
	Furthermore, if~$C_h^0$ is constructed locally on boundary facets and~$I_h$ is chosen appropriately, then also~\ref{prop:trace-pres}$_{\Gamma}$ holds, see Remark~\ref{rmk:properties_Fortin}~\ref{itm:rmk_part-bd}. 
\end{proposition}
\begin{proof}	
 Let us first collect the properties of the operator~$C_h = C_h^0+C_h^1(I-C_h^0) $.    
	By inverse estimates the stability in~\eqref{eq:Chj-stab} implies that 
		\begin{align*}
		\dashint_T |C_h^j v|\dx + h_T \dashint_T |\nabla( C_h^j v)|\dx  \lesssim \dashint_{\Omega_h^{s(j)}(T)} |v| \dx + h_T \dashint_{\Omega_h^{s(j)}(T)} |\nabla v| \dx,
	\end{align*}
	and hence, with~$\bar{s} \coloneqq s(0)+ s(1)$, the operator~$C_h$ satisfies   
	\begin{align}\label{eq:Ch-stab}
		\dashint_T |C_h v|\dx + h_T \dashint_T |\nabla ( C_h v)|\dx  \lesssim \dashint_{\Omega_h^{\bar{s}}(T)} |v| \dx + h_T \dashint_{\Omega_h^{\bar{s}}(T)} |\nabla v| \dx
	\end{align}
	for any~$v\in W^{1,1}(\Omega)^d$, for any $T \in \tria_h$, uniformly in $h$. 
	Furthermore, since~$C_h^1$ maps to~$X_{h,0}$, property~\ref{itm:lemma_trace-int} implies that~$C_h(W^{1,1}_0(\Omega)^d)
	\subset X_{h,0}$, \; $C_h(W^{1,1}_n(\Omega)) \subset X_{h,n} $ and 
	\begin{align}\label{itm:Ch-trace-int}
		\quad 	\int_{f} \tr(v-C_h v)\cdot n \dsig = 0\quad \text{for any } v\in W^{1,1}(\Omega)^d~\text{and all } f\in\facets^\partial(\tria_h).
	\end{align} 
	We proceed to show~\ref{prop:lin-proj}--\ref{prop:local-stab-approx}. 

	\textit{Verification of~\ref{prop:lin-proj}}: Since~$I_h$ and~$C_h$ are linear, so is~$\Pi_h$. 
	The projection property follows, since by Lemma~\ref{lem:SZinterp}~$I_h$ is the identity on $X_h$, and by linearity we have~$C_h(0) = 0$. 
	
	\textit{Verification of~\ref{prop:div-pres}}: The divergence preservation~\ref{prop:div-pres} follows from the construction~\eqref{eq:constr_inhom} and property~\ref{itm:lemma_div-pres}.
	
	\textit{Verification of~\ref{prop:trace-pres}}: 
	Since both~$I_h$ and~$C_h^0$ preserve zero traces and zero normal traces, see  Proposition~\ref{prop:SZinterp}, and since~$C_h^1$ maps to $X_{h,0}$,  it follows that~$\Pi_h$ preserves these traces as well. 
	Hence~\ref{prop:trace-pres}~\ref{itm:global_zerotrace} and~\ref{itm:global_zeronormaltrace} are satisfied. 
	For arbitrary~$v\in W^{1,1}(\Omega)^d$ we set~$w\coloneqq v-I_h v$. 
	By~\eqref{itm:Ch-trace-int} and~\eqref{eq:constr_inhom} we have that
	\begin{align*}
	\int_f \tr((I-\Pi_h)(v))\cdot n\dsig 
	&= \int_f 	\tr((I-C_h)(w))\cdot n \dsig  = 0, 
	\end{align*}
	for any $f \in \facets^\partial(\tria_h)$.  
	Hence, the operator~$\Pi_h$ also satisfies~\ref{prop:trace-pres}~\ref{itm:global_meantrace}. \\
	Furthermore, if $I_h$ is chosen appropriately, see Remark~\ref{rmk:SZ-discr-traces}~\ref{itm:SZ-part-bd}, and $C_h^0$ is local on boundary facets, then also \ref{prop:trace-pres}$_{\Gamma}$ as described in Remark~\ref{rmk:properties_Fortin}~\ref{itm:rmk_part-bd} holds. 

\textit{Verification of~\ref{prop:global-stab-approx} and \ref{prop:local-stab-approx}}: 
In view of Remark~\ref{rmk:properties_Fortin}~\ref{rmk:stab_approx} it suffices to verify the local approximation property \eqref{eq:prop_local-approx}. 
The ingredients are the stability of the operator~$C_h$ in~\eqref{eq:Ch-stab}, and the approximation property of the interpolation operator $I_h$, see Proposition~\ref{prop:SZinterp}~\ref{itm:Ih-local_approx}. 
Let~$p\in[1,\infty]$ and let~$\ell \in \mathbb{N}_0$ and~$m \in \mathbb{N}$  with $\ell \leq m\leq k+1$. 
It follows from~\eqref{eq:Ch-stab} by inverse estimates, Hölder's inequality, scaling properties and the shape regularity of~$(\tria_h)_h$ that 
\begin{align}\label{est:Chst}
		h_T^{\ell} \norm{\nabla^\ell C_h v}_{L^p(T)} \lesssim \norm{v}_{L^p(\Omega_h^{\bar{s}}(T))} + h_T \norm{\nabla v}_{L^p(\Omega_h^{\bar{s}}(T))}
\end{align}
holds for all~$v\in W^{1,p}(\Omega)^d$, for all $T\in \tria_h$, independently of $h$.   
Using the definition~\eqref{eq:constr_inhom} of~$\Pi_h$, \eqref{est:Chst} and Proposition~\ref{prop:SZinterp}~\ref{itm:Ih-local_approx}, and setting~$s\coloneqq 1+\bar{s}$ we obtain  
\begin{align*}
	h_T^\ell \norm{\nabla^\ell ((I-\Pi_h)(v))}_{L^p(T)} 
	&= h_T^\ell \norm{\nabla^\ell \left((I-C_h)(I-I_h)(v)\right)}_{L^p(T)} \\
	&\lesssim \norm{v-I_h v}_{L^p(\Omega_h^{\bar{s}}(T))} + h_T \norm{\nabla(v-I_h v)}_{L^p(\Omega_h^{\bar{s}}(T))} \\
	&\qquad + h_T^\ell \norm{\nabla^\ell (v-I_h v)}_{L^p(T)} \\
	&\lesssim h_T^m \norm{\nabla^m v}_{L^p(\Omega_h^{s}(T))}
\end{align*}
holds for all~$v\in W^{m,p}(\Omega)^d$, for any $T \in \tria_h$, uniformly in $h$.   
Thus, $\Pi_h$ satisfies the desired local approximation property~\eqref{eq:prop_local-approx}, which proves the claim. 
\end{proof}

\begin{remark}\label{rmk:global}
	If one of the correction operators~$C_h^0, C_h^1$ in the construction of~$\Pi_h$ in~\eqref{eq:constr_inhom} is not local, then the arguments in the proof of Proposition~\ref{prop:Chj_operators} yield the existence of a Fortin operator satisfying \ref{prop:lin-proj}--\ref{prop:global-stab-approx}, but not locality. 
\end{remark}

\subsection{Elements with discontinuous pressure space}\label{subsec:discont}
In the following we present constructions of local Fortin operators for finite elements with discontinuous pressure spaces, corresponding to the choice~$\mathcal{Q}=L^{\infty}(\Omega)$ in~\eqref{def:Qh}. 
The construction does not require any modification compared to the homogeneous case, but we verify all properties~\ref{prop:lin-proj}--\ref{prop:local-stab-approx}. 
 
\subsubsection{Bernardi--Raugel element}\label{subsec:Bernardi-Raugel-v}

In the homogeneous case for $d\in\{2,3\}$, the Bernardi--Raugel element  was originally presented in~\cite{BernardiRaugel1985}, see also \cite[Lem.~5.4.5]{Bernardi2024}.
It is a special element because the quasi-interpolation operator of Scott--Zhang type with respect to its degrees of freedom (cf.~Lemma~\ref{lem:SZinterp}) is already a Fortin operator satisfying all properties~\ref{prop:lin-proj}--\ref{prop:local-stab-approx}.

In the lowest-order case (${k} = 1$)  for $d\in\{2,3\}$, see \cite[Sec.~II]{BernardiRaugel1985}, the velocity space $X_h$ is given by~\eqref{def:Xh} with $X(T) = \mathcal{P}_1(T)^d + B_n(T)$, where $B_n(T)$ is the local facet bubble space defined in \eqref{def:local_normal_face_bubblespace}. 
The pressure space $Q_h$ as in \eqref{def:Qh} consists of discontinuous functions with~$\mathcal{Q}=L^\infty(\Omega)$ and $Q(T)=\mathcal{P}_0(T)$. 

\begin{remark}\label{rmk:BR-global} 
	The global spaces of the lowest-order Bernardi--Raugel element for~$d \in \{2,3\}$ are $$Q_h=\mathcal{L}_0^0(\tria_h)\qquad \text{and}\qquad	X_h = \mathcal{L}_1^1(\tria_h)^d + B_n(\tria_h),$$ with facet bubble space  $B_n(\tria_h)\subset \mathcal{L}_d^1(\tria_h)^d$  as in \eqref{def:global_normal_face_bubblespace}. 
	The pair~$(X_h,Q_h)$  satisfies Assumption~\ref{ass:FEM} with~$k = 1$. 
	Its degrees of freedom can be chosen as those of~$\mathcal{L}^1_1(\tria_h)^d$, together with the additional degrees of freedom, for $v \in W^{1,1}(\Omega)^d$, 
	\begin{align}
		v \mapsto \dashint_f v \cdot n_f \dsig  \qquad \text{ for any } f \in \facets(\tria_h).
	\end{align} 
\end{remark}

\begin{lemma}[lowest-order Bernardi--Raugel]\label{lem:BR_Fortin-v}
	For the lowest-order Bernardi--Raugel element the operator~$\Pi_h\colon W^{1,1}(\Omega)^d \to X_h$ in~\eqref{eq:constr_inhom} with $I_h$ chosen as in Lemma~\ref{lem:SZinterp} with respect to the degrees of freedom in Remark~\ref{rmk:BR-global}, and with $C_h^1\equiv C_h^0 \equiv 0$
    is a Fortin operator satisfying properties~\ref{prop:lin-proj}--\ref{prop:local-stab-approx} with~$k = 1$ and $s = 1$. 
	By suitable choice of $I_h$ also~\ref{prop:trace-pres}$_{\Gamma}$ can be satisfied, see Remark~\ref{rmk:properties_Fortin}~\ref{itm:rmk_part-bd}. 
\end{lemma}
\begin{proof}
 Using corresponding local `Lagrange' basis functions with respect to the degrees of freedom\footnote{By this we mean a basis for which evaluation at its own degree of freedom yields one and at the other degrees of freedom it yields zero.} in Remark~\ref{rmk:BR-global}, the construction in~\cite{ScottZhang1990} shows the projection property~\ref{prop:lin-proj}, as well as approximation and stability properties~\ref{prop:local-stab-approx}, and hence also~\ref{prop:global-stab-approx}. 
	Noting that the coefficients for the facet degrees of freedom are simply~$\dashint_f v \cdot n_f \dsig$, it follows from the definition  of~$I_h$ that 
	\begin{align}
	\int_f (I_h v)\cdot n_f \dsig = \int_f v \cdot n_f \dsig \qquad \text{for all }f\in\facets(\tria_h), \label{eq:constrBernardi-Raugel-Ih}
\end{align}
for any $v\in W^{1,1}(\Omega)^d$. 
This ensures the trace-preservation properties involving the normal traces \ref{prop:trace-pres}~\ref{itm:global_zeronormaltrace} and \ref{itm:global_meantrace}. 
For~$v \in W^{1,1}_0(\Omega)^d$ all coefficients of boundary degrees of freedom are zero and hence property \ref{prop:trace-pres}~\ref{itm:global_zerotrace} is satisfied. 
	 
Thus, it suffices to show~\ref{prop:div-pres}. Indeed, using~\eqref{eq:constrBernardi-Raugel-Ih}, for any  piecewise constant pressure function $q_h\in\overline{Q}_h$ we have that 
		\begin{align*}
			\int_T \divergence((I-I_h)(v)) q_h\dx 
			= q_h\big|_T  \int_{\partial T}(I-I_h)(v)\cdot n\dsig = 0,
		\end{align*} 
		for any $v\in W^{1,1}(\Omega)^d$ and for any $T \in \tria_h$.
		Hence $\Pi_h$  preserves the divergence in the dual of $\overline{Q}_h$ locally on each $T \in \tria_h$, and thus also globally. 
		This shows the claim.   
\end{proof}

Since in the following it plays an important role, let us introduce also the operator mapping only to the normal facet bubbles. 
It may be used in an alternative construction of the Fortin operator for the Bernardi--Raugel element. 
Recall the space~$B_n(\tria_h)$ defined in~\eqref{def:global_normal_face_bubblespace}. 

\begin{lemma}[operator $C_h^0$] 
	\label{lem:Ch0}
	There is a linear operator 
	$C_h^0\colon W^{1,1}(\Omega)^d \to B_n(\tria_h)$, which satisfies  $C_h^0(W^{1,1}_0(\Omega)^d)
		\subset W^{1,1}_{0}(\Omega)^d$,  $C_h^0(W^{1,1}_n(\Omega)) \subset W_{n}^{1,1}(\Omega)^d $ and  
	\begin{align}
		\int_f (C_h^0 v)\cdot n_f \dsig = \int_f v \cdot n_f \dsig \qquad \text{for all }f\in\facets(\tria_h), \label{eq:constrBernardi-Raugel}
	\end{align}
	for any $v\in W^{1,1}(\Omega)^d$. Furthermore, it satisfies 
	\begin{align}\label{eq:BR_Chdiv-stab}
		\norm{C_h^0  v}_{L^1(T)} \lesssim \norm{v}_{L^1(T)} + h_T \norm{\nabla v}_{L^1(T)}\qquad \text{for any } v\in W^{1,1}(\Omega)^d,
	\end{align}
	for any $T \in \tria_h$, uniformly in $h$. 
	It is local in the sense that 
	$\support(C_h^0 v) \subset \overline{\Omega_{h}(f)}$ for any $v \in W^{1,1}(\Omega)^d$ with   $\support(v) \subset \overline{\Omega_{h}(f)}$, for any $f \in \facets(\tria_h)$. 
\end{lemma}
\begin{proof}
	By construction it is immediate that the operator is uniquely determined and that it is local. 
	One can show that it is locally $W^{1,1}$-stable, see for example~\cite[p.~174]{Tscherpel2018}. 
\end{proof}

\begin{remark}\label{rmk:Ch0}
	The operator $C_h^0$ satisfies the properties in Proposition~\ref{prop:Chj_operators} with $s(0) = 0$ in~\ref{itm:lemma_stab_approx} for general finite element pairs $(X_h,Q_h)_h$, provided that~$B_n(\tria_h) \subset X_h$.
	Hence, in this case, to construct a Fortin operator with properties~\ref{prop:lin-proj}--\ref{prop:local-stab-approx} only the properties of~$C_h^1$ and the  divergence preservation in Proposition~\ref{prop:Chj_operators}~\ref{itm:lemma_div-pres} have to be verified. 
\end{remark}
 
 \begin{remark}\label{rmk:BR-Ch0}
If the projection property of ~$\Pi_h$ onto~$X_h$ or the discrete trace preservation in Remark~\ref{rmk:properties_Fortin}~\ref{itm:rmk_discr-trace}  is not important, one can also construct the Fortin operator for the lowest-order Bernardi--Raugel element by using the standard Scott--Zhang operator~$I_h$ mapping to~$\mathcal{L}^1_1(\tria_h)^d$ and using the correction operator~$C_h^0$ as in Lemma~\ref{lem:Ch0} and~$C_h^1 \equiv 0$. 
This is the approach taken in~\cite[Lem.~5.4.5]{Bernardi2024}. 
 \end{remark}

\begin{remark}\label{rmk:BR-2ndorder}
	The second order Bernardi-Raugel element $(k = 2)$  for $d=3$ was introduced in~\cite[Sec.~III]{BernardiRaugel1985}. 
	The pressure space $Q_h$ in~\eqref{def:Qh} arises by choosing~$\mathcal{Q}=L^\infty(\Omega)$ and $Q(T)=\mathcal{P}_1(T)$ for all $T\in\tria_h$. 
	The velocity space $X_h$ is defined as  in~\eqref{def:Xh} with local velocity space~$X(T)=\mathcal{P}_2(T)^3 + B_n(T) + B(T)^3$. 
	A Fortin operator for the homogeneous case and for continuous functions was constructed in~\cite[Lem.~III.1]{BernardiRaugel1985}. 
	It can be adapted to $W^{1,1}(\Omega)$-functions with a Scott--Zhang-type construction, and  it can be directly extended to the  inhomogeneous case. 
	
	For a simplex~$T=[z_0,z_1,z_2,z_3]\in\tria_h$ the degrees of freedom of the extra element bubble space~$B(T)^3$, see~\eqref{def:local_element_bubblespace}, can be chosen as 
	\begin{align}\label{eq:dof_2ndBR}
		v\mapsto \int_T \varphi_{z_i} \divergence v \dx \qquad \text{for }i\in\{1,2,3\},
	\end{align}
	for~$v\in W^{1,1}(\Omega)^3$ with hat function $\varphi_z\in\mathcal{L}_1^1(\tria_h)$ for $z\in\vertices(\tria_h)$ defined in \eqref{def:hatfunc}. 
	This suffices, since $\int_{T} 1 \divergence v\dx = \int_{\partial T} v \cdot n \dsig$ is already determined by the normal face bubble degrees of freedom. 
	Then, analogously to the lowest-order case, the quasi-interpolation operator $I_h$ in Lemma~\ref{lem:SZinterp} with respect to the degrees of freedom in Remark~\ref{rmk:BR-global} and \eqref{eq:dof_2ndBR} is a Fortin operator satisfying all properties \ref{prop:lin-proj}--\ref{prop:local-stab-approx} with $k=2$ and $s=1$. 
\end{remark}

 \subsubsection{$P_d-P_0$ element}\label{subsec:Pd-P0}
 	 
\begin{remark}{($P_d-P_0$ element)}\label{rmk:PdP0}
The $P_d-P_0$ element for dimension $d \in \{2,3\}$ uses the discontinuous pressure space as in~\eqref{def:Qh} with $\mathcal{Q}=L^\infty(\Omega)$ and $Q(T) = \mathcal{P}_0(T)$; for $d = 2$ see~\cite[Sec.~8.4.3]{Boffi2013}. 
Its velocity space is $X_h = \mathcal{L}^1_d(\tria_h)^d$, i.e., it satisfies Assumption~\ref{ass:FEM} with $k = d$. 
Hence, the lowest-order Bernardi--Raugel element can be interpreted as  a reduced $P_d-P_0$ element. 
For this reason, the Fortin operator of the Bernardi--Raugel element in Lemma~\ref{lem:BR_Fortin-v}  is trivially a Fortin operator for the $P_d-P_0$ element, but this operator has reduced approximation order (for $k = 1$).  
Choosing~$I_h$ instead as the standard Scott--Zhang operator mapping to~$X_h = \mathcal{L}^1_d(\tria_h)^d$ and correcting with $C_h^0$ in Lemma~\ref{lem:Ch0} and $C_h^1 \equiv 0$, one obtains \ref{prop:lin-proj}--\ref{prop:local-stab-approx} with~$k = d$, cf. Remark~\ref{rmk:BR-Ch0}. 
				
Note that the local Fortin operator for the $P_2-P_0$ element for~$d=2$ introduced in~\cite[Sec.~8.4.3]{Boffi2013} for the homogeneous case uses full facet traces instead of normal traces. 
It fits without modification into the framework presented in Section~\ref{sec:main_results} and satisfies all properties. 
\end{remark}
		
\begin{remark}
In dimension~$d =2$ the $P_2-P_0$ element can be generalised to a~$P_{k}-P_{k-2}$ element for~$k\geq 2$, meaning that $X(T) = \mathcal{P}_{k}(T)^2$ and~$Q(T) = \mathcal{P}_{k-2}(T)$. 
This is an enriched version of the conforming Crouzeix--Raviart element of order $k-1$, see Remark~\ref{rmk:CR-global} below. 
For this reason, one may construct a Fortin operator using an approach similar to that in Section~\ref{subsec:confCR} below with minor modifications. 
One would have to replace~$I_h$ by the Scott--Zhang operator mapping to~$\mathcal{L}^1_{k}(\tria_h)^2$, and choose the operators~$C_h^0$ and~$C_h^1$ as in Proposition~\ref{prop:Fortin_CR}.    	
\end{remark}

\subsubsection{Conforming Crouzeix--Raviart element}\label{subsec:confCR} 
The conforming Crouzeix--Raviart element was introduced in~\cite{CrouzeixRaviart1973} for $d\in\{2,3\}$ in its lowest-order version $k = 2$. 
It was extended to arbitrary order $k\geq 2$ in~\cite{Mansfield1982}. 
See also~\cite[Sec.~5.4.1.4 for $d=2$; Sec.~5.4.1.5 for $d=3$]{Bernardi2024}.  

The velocity space $X_h$ is defined as in~\eqref{def:Xh} with 
\begin{align}
	X(T) = \mathcal{P}_k(T)^d + B_{k-2,d}(T)^d,
\end{align} 
where $B_{k-2,d}(T)\subseteq \mathcal{P}_{k+d-1}(T)$ is defined in \eqref{def:Bk-2d}. 
The pressure space $Q_h$ is given as in \eqref{def:Qh} with $\mathcal{Q}=L^\infty (\Omega)$ and $Q(T)= \mathcal{P}_{k-1}(T)$.

\begin{remark}\label{rmk:CR-global}
	The global spaces of the generalised conforming Crouzeix--Raviart element are 
	$$Q_h=\mathcal{L}_{k-1}^0(\tria_h)\qquad  \text{and} \qquad X_h = \mathcal{L}_k^1(\tria_h)^d + B_{k-2,d}(\tria_h)^d$$ 
	with $B_{k-2,d}(\tria_h) \subseteq \mathcal{L}_{k+d-1}^1(\tria_h)$ as in~\eqref{def:global_Bk-2d}. 
\end{remark}

In the original work~\cite[Sec.~4]{CrouzeixRaviart1973} the authors construct a local Fortin operator for~$d \in \{2,3\}$ and~$k = 2$, for functions in $W^{2,2}(\Omega)^d$. 
Subsequently, in~\cite[Sec.~3.3]{Girault2003} Girault and Scott construct a local Fortin operator in dimension~$d=2$ for velocity functions with zero trace in the lowest-order case $k=2$ and state that its construction extends to higher polynomial degree. 
The extension to $d \in \{2,3\}$ and $k \geq d$ was presented in~\cite[Sec.~5.4.1.4, 5.4.1.5]{Bernardi2024}. 
However, their operator is not constructed to be a projection onto $X_{h}$, but only a projection onto~$\mathcal{L}^1_k(\tria_h)^d$. 
As for the Bernardi--Raugel element, by modifying the construction in~\cite{CrouzeixRaviart1973} using a quasi-interpolation operator one may obtain a Fortin operator applicable also to functions with nonzero trace. 
	 
To avoid the difficulty of choosing the degrees of freedom for higher order, we proceed by the correction approach in~\eqref{eq:constr_inhom} to establish a Fortin operator for~$k\geq d$ and dimension~$d\geq 2$. 
This resembles the construction in~\cite{Girault2003}, but yields a projection onto~$X_h$. 
Then we show that this operator indeed satisfies properties~\ref{prop:lin-proj}--\ref{prop:local-stab-approx}. 

The operator~$C_h^0$ in~\eqref{eq:constr_inhom} is the same as the one used for the Bernardi--Raugel element, see Lemma~\ref{lem:Ch0}. 
Using this operator is possible, since for~$k \geq d$ we have~$B_n(\tria_h) \subset \mathcal{L}^{1}_{d}(\tria_h)^d \subset X_h$. 

Let us now specify the operator~$C_h^1 \colon W^{1,1}(\Omega)^d \to B_{k-2,d}(\tria_h)^d \subset X_{h,0}$, with~$B_{k-2,d}(\tria_h)$  as defined in~\eqref{def:global_Bk-2d}. 
Since~$B_{k-2,d}(\tria_h)$ has a basis of functions supported on a single simplex~$T \in \tria_h$ the construction is local (per simplex). 
Let $T\in\tria_h$ be a $d$-simplex, and $\Phi_T \colon \hat T \to T$ an affine transformation as in~\eqref{def:transform}. 
Let~$(\hat q_i)_{i\in \{1, \ldots, \delta\}}$ denote a fixed basis of $\mathcal{P}_{k-1}(\hat T)\cap L_0^2(\hat T)$, where 
	\begin{align}
		\delta \coloneqq \delta(k) \coloneqq \dim \mathcal{P}_{k-1}(\hat T) - 1 = \begin{pmatrix}
			k+d-1 \\ d\end{pmatrix} - 1
\end{align} 
Then, $( q_i)_{i\in \{1, \ldots, \delta\}}$, defined by $ q_i = \hat q_i \circ \Phi_T^{-1}$, forms a basis of~$\mathcal{P}_{k-1}(T)\cap L_0^2(T)$. 
The fact that $b_T \nabla q_i \in B_{k-2,d}(T)^d$  for $i \in \{1, \ldots, \delta\}$, with $B_{k-2,d}(T)^d$ as in~\eqref{def:Bk-2d}, motivates the following construction: 
We define $C_{h,T}^1\colon W^{1,1}(\Omega)^d \to \linearspan\{b_T\nabla q\colon  q\in \mathcal{P}_{k-1}(T)\cap L_0^2(T)\}$ determined by
\begin{alignat}{2}
	\int_T C_{h,T}^1\, v \cdot \nabla q_j \dx &= -\int_T  \divergence v ~q_j\dx \qquad &&\text{for all }j \in\{1,...,\delta \},
	\label{eq:constrCR}
\end{alignat} 
for $v \in W^{1,1}(\Omega)^d$.  
This equality is equivalent to determining the coefficient vector $\alpha \in \setR^\delta$ as the solution to the linear system of equations $M\alpha = F$, where
\begin{align}\label{eq:CR_lin-system} 
M_{i,j} \coloneqq \dashint_T b_T\, \nabla q_i\cdot\nabla q_j\dx, \qquad F_j\coloneqq -\dashint_T \divergence v~ q_j\dx,\quad i,j\in\{1,...,\delta\},
\end{align}
and setting 
\begin{align}\label{eq:constr_CR-lincomb}
	C_{h,T}^1 \,v\coloneqq \sum_{i=1}^{\delta} \alpha_i b_T \nabla q_i. 
\end{align}
Well-posedness of the operator $C_{h,T}^1$ follows from invertibility of $M$, which we shall establish below.  
Since~$C_{h,T}^1\, v \big|_{\partial T} = 0$ for each $T \in \tria_h$, we extend the functions by zero to $\Omega$, and define the global operator $C_h^1 \colon W^{1,1}(\Omega)^d \to B_{k-2,d}(\tria_h)^d \subset X_{h,0}$ as
 \begin{align}\label{def:CR_Ch1-global}
	C_h^1 v \coloneqq \sum_{T\in\tria_h} C_{h,T}^1 v \qquad \text{ for any } v \in W^{1,1}(\Omega)^d,
\end{align} 
with~$C_{h,T}^1$ as in~\eqref{eq:constrCR}.  
 By construction in~\eqref{eq:constrCR} the operator~$C_h^1$  satisfies
 \begin{align}
 	\int_T C_h^1 v \cdot \nabla q_h \dx = -\int_T  \divergence v ~q_h\dx  \label{eq:constrCR-global}
 \end{align}
 for any $T \in \tria_h$, and for any $q_h \in Q_h$ satisfying~$\dashint_{T} q_h \dx = 0$ on every simplex~$T \in \tria_h$.
   
The following lemma establishes well-posedness of $C_{h,T}^1$ as defined in~\eqref{eq:constrCR}. 
It is also used to prove local~$W^{1,1}(\Omega)$-stability of $C_h^1$ in the sequel. 
  
\begin{lemma}\label{lem:CR-matrix}
	For a simplex~$T\in\tria_h$ the corresponding matrix $M\in\R^{\delta \times \delta}$ defined in~\eqref{eq:CR_lin-system} is symmetric and positive definite. 
	Furthermore, there is a constant $C>0$ (depending only on the shape-regularity constant, $k$ and $d$) such that 
	\begin{align*}
	\norm{M^{-1}}_2 \leq C h_T^{2},
	\end{align*} 
	for any $T\in\tria_h$ and uniformly in $h$.
\end{lemma}
\begin{proof}
	The matrix $M$ is symmetric by definition in~\eqref{eq:CR_lin-system}. 
	To show that it is positive definite let $\alpha\in\R^{\delta}\backslash\{0\}$ be arbitrary and set~$q_h\coloneqq \sum_{i=1}^{\delta} \alpha_i q_i \in \mathcal{P}_{k-1}(T)\cap L_0^2(T)$. 
	We note that~$\nabla q_h \not\equiv 0$, because~$q_h$ is not constant. 
	Then, we have 
	\begin{align}\label{eq:CR_M-pos-def}
		\alpha^\top M \alpha = \dashint_T b_T |\nabla q_h|^2 \dx > 0\qquad \text{for any } \alpha\neq 0,
	\end{align}
	and hence $M$ is positive definite. 
	Denoting by~$\lambda_{\min}>0$ the smallest eigenvalue of $M$, the spectral norm of $M^{-1}$ is $\norm{M^{-1}}_2 = \frac{1}{\lambda_{\min}}$. 
	Thus, it suffices to show that 
	\begin{align*}
		\alpha^\top M \alpha \gtrsim h_T^{-2} \abs{\alpha}^2\qquad \text{for all }\alpha\in\R^{\delta}\backslash\{0\}.
	\end{align*} 
	We consider the rescaled matrix $\hat M \in \mathbb{R}^{\delta \times \delta}$ on the reference simplex $\hat T$ with the entries 
		\begin{align*}
			\hat M_{i,j} \coloneqq \dashint_{\hat T} \hat b_{\hat{T}}\, \hat \nabla \hat{q}_i\cdot \hat \nabla \hat{q}_j\, \mathrm{d}\hat{x} \qquad \text{ for } i,j \in \{1, \ldots, \delta\},
	\end{align*}  
	with rescaled element bubble $\hat{b}_{\hat{T}}$ and pressure functions $(\hat{q}_i)_{i \in \{1, \ldots, \delta\}}$  and $\hat{q}_h$. 
	Then, by scaling arguments, see~\eqref{eq:scaling}, we have 
	\begin{align*}
		\alpha^\top M \alpha 
		= 
		\dashint_T b_T |\nabla q_h|^2 \dx  
		\eqsim 
		h_T^{-2} \dashint_{\hat{T}} \hat{b}_{\hat{T}} |\hat \nabla \hat{q}_h|^2 \,\mathrm{d}\hat{x}
		 = 
		 h_T^{-2}	\alpha^\top \hat M \alpha, 
	\end{align*}
	with constants depending on the shape-regularity constant, on $k$ and $d$.  
	Since $\hat{M}$ is symmetric and positive definite with the same arguments as $M$, and its smallest eigenvalue  depends only on $k,d$, this proves the claim. 
\end{proof}

Now we are in the position to show well-posedness and local stability of the operator~$C_h^1$. 

\begin{lemma}[properties of~$C_h^1$]\label{lem:Ch1_CR}
	The linear operator~$C_h^1\colon W^{1,1}(\Omega)^d\to B_{k-2,d}(\tria_h)^d \subset X_{h,0}$ as defined in~\eqref{def:CR_Ch1-global} is well-defined. 
	It satisfies  
	\begin{align}
		\int_T C_h^1 v \cdot \nabla q_h \dx = -\int_T  \divergence v ~q_h\dx  \label{eq:constrCR-global-v}
	\end{align}
	for any $T \in \tria_h$ and for any function $q_h \in Q_h$ satisfying $\dashint_{T} q_h \dx = 0$ on every simplex~$T \in \tria_h$.
	
	Furthermore, it is locally~$W^{1,1}$-stable in the sense that  
	\begin{align}\label{est:W11-CR}
		\norm{C_h^1 v}_{L^1(T)} \lesssim	h_T\norm{\nabla v}_{L^1(T)} \qquad\text{for all }v\in W^{1,1}(\Omega)^d,~\text{for all }T\in\tria_h,
	\end{align}	 
	 uniformly in $h$.
\end{lemma}
\begin{proof}
	By Lemma~\ref{lem:CR-matrix} the matrix~$M$ as in~\eqref{eq:CR_lin-system} is invertible, and hence for each $T \in \tria_h$ the operator  $C_{h,T}^1 \colon W^{1,1}(\Omega)^d \to \linearspan\{b_T\nabla q_h\colon  q_h\in \mathcal{P}_{k-1}(T)\cap L_0^2(T)\}$  satisfying~\eqref{eq:constrCR} is well-defined. 
	The property~\eqref{eq:constrCR-global-v} follows directly by construction, see~\eqref{eq:constrCR-global} and \eqref{eq:constrCR}. 
	 	 
	To show the local stability in~\eqref{est:W11-CR}, let an arbitrary simplex~$T \in \tria_h$ be fixed. 
	For the basis $(q_i)_{i \in \{1, \ldots, \delta\}}$ of $\mathcal{P}_{k-1}(T)\cap L_0^2(T)$ we use the
	scaling properties 
	\begin{align}\label{est:scaling-1}
		\norm{ q_i}_{L^{\infty}(T)}  + 	h_T \norm{\nabla q_i}_{L^{\infty}(T)} + \dashint_{T} b_T \dx \lesssim 1,
	\end{align}
	for any $i \in \{1, \ldots, \delta\}$ with a constant independent of $T$.   
	Applying this to~$C_h^1 v\big|_T = C_{h,T}^1\, v$ as in~\eqref{eq:constr_CR-lincomb} in combination with the shape regularity and the equivalence of norms on~$\setR^\delta$ we obtain 
	\begin{align*}
		\dashint_T \abs{C_{h,T}^1 v}\dx
		\leq \sum_{i=1}^{\delta} \abs{\alpha_i} \norm{\nabla q_i}_{L^\infty(T)}  \dashint_T b_T\dx 
		\lesssim h_T^{-1} \abs{\alpha}, 
	\end{align*}
with constants independent of $T$.  
Using~$M \alpha = F$, with $M$ and $F$ as in \eqref{eq:CR_lin-system}, the estimate on $\norm{M^{-1}}_2$ in Lemma~\ref{lem:CR-matrix}, the scaling estimate~\eqref{est:scaling-1} and equivalence of norms on~$\setR^\delta$  we obtain     
\begin{align*}
	\dashint_T \abs{C_{h,T}^1 v}\dx
	&\lesssim h_T^{-1} \abs{\alpha}
 \lesssim h_T^{-1} \norm{M^{-1}}_{2} \abs{F} \lesssim h_T \max_{j \in \{1,\ldots, \delta\}} \abs{F_j}
 	\\
 	&\lesssim h_T \dashint_{T}\abs{\divergence v} \dx  \max_{j\in\{1,\dots, \delta\}}\norm{q_j}_{L^{\infty}(T)}
 	 \lesssim h_T \dashint_{T}\abs{\nabla v} \dx. 
 \end{align*}		

The hidden constant depends only on the shape-regularity constant,~$d$ and on~$k$.  
\end{proof}

As before, we use Proposition~\ref{prop:Chj_operators} to verify the properties of the Fortin operator for the Crouzeix--Raviart element, as introduced in Remark~\ref{rmk:CR-global}.

\begin{proposition}[conforming Crouzeix--Raviart]\label{prop:Fortin_CR}
	For the conforming Crouzeix--Raviart element for~$d \in \{2,3\}$ and~$k \geq d$  the operator~$\Pi_h\colon W^{1,1}(\Omega)^d \to X_h$ defined by~\eqref{eq:constr_inhom} with~$I_h$ as in Lemma~\ref{lem:SZinterp},  with $C_h^0$ as in Lemma~\ref{lem:Ch0}, and with~$C_h^1$ as in Lemma~\ref{lem:Ch1_CR} is a Fortin operator satisfying the properties \ref{prop:lin-proj}--\ref{prop:local-stab-approx} with~$s = 1$. 
	By suitable choice of $I_h$ also~\ref{prop:trace-pres}$_{\Gamma}$ can be satisfied, see Remark~\ref{rmk:properties_Fortin}~\ref{itm:rmk_part-bd}.	
\end{proposition}
\begin{proof}
	We verify the assumptions of Proposition~\ref{prop:Chj_operators}.
	Recall that the operator~$C_h^0$, see Lemma~\ref{lem:Ch0}, maps to $B_n(\tria_h)$ and for~$k \geq d$ one has~$B_n(\tria_h) \subset \mathcal{L}^{1}_{d}(\tria_h)^d \subset X_h$. 
	By Lemma~\ref{lem:Ch0} and by Lemma~\ref{lem:Ch1_CR} on the properties of $C^1_h$, it remains only to 
	verify the divergence-preservation property in Proposition~\ref{prop:Chj_operators}~\ref{itm:lemma_div-pres}, see  Remark~\ref{rmk:Ch0}. 
	
Let $q_h \in \overline{Q}_h$ be arbitrary but fixed. 
We decompose $q_h$ into a piecewise constant function with values $\langle q_h\rangle_T$ on $T \in \tria_h$, and the locally mean-free function  
\begin{align}\label{def:qhtilde}
\tilde{q}_h\coloneqq q_h - \sum_{T\in\tria_h}  \mathbb{1}_T \langle q_h\rangle_T,
\end{align} 
i.e., $\tilde{q}_h\big|_T\in\mathcal{P}_{k-1}(T) \cap L^2_0(T)$. 
By property~\eqref{eq:constrCR-global-v} of $C_h^1$ in Lemma~\ref{lem:Ch1_CR} for locally mean-free pressure functions and  integrating by parts we have for any $w \in W^{1,1}(\Omega)^d$ and any $T \in \tria_h$ that  
\begin{equation}\label{eq:Fortin-CR-div-1}
\begin{aligned}
0 & = \int_T \divergence((I-C_h^0)(w)) \tilde{q}_h \dx + \int_T C_h^1(I-C_h^0)(w) \cdot  \nabla \tilde{q}_h \dx \\
& = \int_T \divergence((I-C_h^1)(I-C_h^0)(w)) \tilde{q}_h \dx,
\end{aligned}
\end{equation}   
where the boundary terms vanish thanks to $C_h^1(v)\big|_{\partial T} =0$ for any $v \in W^{1,1}(\Omega)^d$. 
Using the definition of~$C_h= C_h^0 + C_h^1(I-C_h^0)$, the decomposition of $q_h|_{T} = \tilde q_h|_{T} + \langle q_h \rangle_{T}$ as in~\eqref{def:qhtilde}, equation~\eqref{eq:Fortin-CR-div-1}, the fact that~$C_h^1(v)\big|_{\partial T} =0$ for any $v \in W^{1,1}(\Omega)^d$ and the property of $C_h^0$ in~\eqref{eq:constrBernardi-Raugel} we find that 
 	\begin{align*}
 	\int_{T} \divergence((I-C_h)(w)) q_h \dx &= \int_T  \divergence((I-C_h^1)(I-C_h^0)(w)) q_h \dx \\
 	&= \int_T \divergence((I-C_h^1)(I-C_h^0)(w)) \tilde{q}_h \dx \\
 	&\qquad +  \langle q_h \rangle_T \int_T \divergence((I-C_h^1)(I-C_h^0)(w))  \dx \\
 	&=
  \langle q_h \rangle_T  \int_{\partial T} (I-C_h^0)(w)\cdot n \dsig = 0,
 \end{align*}
 for any $T \in \tria_h$ and any $w \in W^{1,1}(\Omega)^d$.  
 Summing over $T \in \tria_h$ shows that Proposition~\ref{prop:Chj_operators}~\ref{itm:lemma_div-pres} holds and hence completes the proof. 
 \end{proof}

\begin{remark}\label{rmk:confCRd3k2}
	The case $d = 3$ and $k = 2$ is not covered by Proposition~\ref{prop:Fortin_CR}. 
	An alternative approach is presented in~\cite[Sec.~III]{BernardiRaugel1985}, see also~\cite[Sec.~5.4.1.5]{Bernardi2024}. 
	It consists of enriching the velocity space with the facet bubble functions~$B_n(T)$, as defined in~\eqref{def:local_normal_face_bubblespace}, which are used in the lowest-order Bernardi--Raugel element. The local velocity space is  
	\begin{align*}
	X(T)=\mathcal{P}_2(T)^3 + B_{0,d}(T)^3 + B_n(T).
	\end{align*}
	Recall that~$B_{0,d}(T) = B(T)$, as defined in~\eqref{def:local_element_bubblespace}.   
	Then, the global spaces are 
	\begin{align*}
		Q_h = \mathcal{L}^0_1(\tria_h)
		\qquad \text{ and } \qquad 
		X_h = \mathcal{L}^1_2(\tria_h)^3 + B(\tria_h)^3 + B_n(\tria_h),
	\end{align*}
	see~\eqref{def:global_el_bubblespace} and~\eqref{def:global_normal_face_bubblespace} for the global bubble spaces.  
	Then, the Fortin operator in~\eqref{eq:constr_inhom} is based on~$I_h$ mapping to~$X_h$, see Lemma~\ref{lem:SZinterp}, $C_h^0$ as in Lemma~\ref{lem:Ch0} and~$C_h^1$ as in Lemma~\ref{lem:Ch1_CR} for~$k = 2$.  
	It satisfies properties~\ref{prop:lin-proj}--\ref{prop:local-stab-approx} with~$s = 1$ and can be constructed to satisfy also~\ref{prop:trace-pres}$_{\Gamma}$.   
\end{remark}

\subsubsection{Guzmán--Neilan element}\label{subsec:Guzman-Neilan}
	The Guzmán--Neilan element introduced for $d=2$ in \cite{Guzman2014a} and for $d=3$ in \cite{Guzman2014b} is a conforming finite element pair with discontinuous pressure functions that uses rational functions for the velocity space and has exact divergence constraints.
	Its construction is based on local modifications of $H(\divergence;\Omega)$-conforming finite elements that enforce tangential continuity, and hence $H^1(\Omega)$-conformity. 
	We denote  
	\begin{align*}
		H(\divergence;\Omega)\coloneqq \{v\in L^2(\Omega)^d\colon \divergence v\in L^2(\Omega)\},\quad H_0(\divergence;\Omega)\coloneqq\{v\in H(\divergence;\Omega)\colon v\cdot n=0 \text{ on }\partial\Omega\}.
	\end{align*}
	As for the previously mentioned elements, the construction of the local Fortin operator for the homogeneous case in~\cite{Guzman2014a} uses a correction of an interpolation operator as in~\eqref{eq:constr_inhom}. 
To highlight the connection, let us briefly review the construction of~\cite{Guzman2014a} in the special case of lowest order $k=1$ in dimension $d=2$ for functions in $H^1_0(\Omega)^2$. 
 The Fortin operator is of the form 
	\begin{align*}
		(I-\Pi_h) = (I-C_{R})(I-C_{M})(I-I_h),
	\end{align*}
	where $I_h$ is the Scott--Zhang operator mapping to $\mathcal{L}^1_1(\tria_h)^2$, see Proposition~\ref{prop:SZinterp}. 
	The operator $C_{M}\colon H_0^1(\Omega)^2\to M_h\subseteq H_0(\divergence;\Omega)$ maps to the space of curls of certain polynomial edge bubble functions, see~\cite[eq.~(3.15)]{Guzman2014a}.	
	This differs from our framework since $M_h$ is not $H^1(\Omega)$-conforming. 
	Then, the operator~$C_{{R}}\colon  M_h + H_0^1(\Omega)^2  \to {R}_h$ corrects the non-conformity on the interior edges. 
	It maps to a space $R_h$ consisting of curls of certain rational bubble functions, see~\cite[eq.~(3.16)]{Guzman2014a}.
	
	This definition of the Fortin operator $\Pi_h$ extends to $H^1(\Omega)$ without modification.  
	The trace-preservation property~\ref{prop:trace-pres} of $\Pi_h$ follows from using boundary degrees of freedom for all operators, since $C_{M}$ is based on the mean of the normal traces on edges, and $C_R$ uses the mean of tangential traces on edges as degrees of freedom, see~\cite[eq.~(3.17), (3.18)]{Guzman2014a}.  	
In fact, it can be shown that the Fortin  operator~$\Pi_h\colon H^1(\Omega)^2 \to  \mathcal{L}^1_1(\tria_h)^2 + M_h + R_h$ satisfies all properties \ref{prop:lin-proj}--\ref{prop:local-stab-approx}, and~\ref{prop:trace-pres}$_{\Gamma}$ can also be established by suitable choice of $I_h$.  
The local approximation property in~\ref{prop:local-stab-approx} for $p = 2$ has been established in~\cite[Lem.~3.3]{Guzman2014a}.

\subsection{Elements with continuous pressure space}\label{subsec:cont}
Now we consider mixed finite element pairs with $\mathcal{Q}=C(\overline{\Omega})$ in the definition of the pressure space~$Q_h$ in~\eqref{def:Qh}. 

\subsubsection{Modified MINI element}\label{subsec:modMINI}
The MINI element was first introduced for~$d=2$ for lowest order $k = 1$  in~\cite{ArnoldBrezziFortin1984}, 
and has generalisations to~$d\in\{2,3\}$ and~$k \geq 1$, see~\cite[Sec.~8.4.2, 8.7.1]{Boffi2013}. 

The main reason to use it is that the lowest-order version has fewer degrees of freedom than the previously mentioned elements and the Taylor--Hood element, except for the lowest-order Bernardi--Raugel element in dimension~$d  = 3$.  
Note, however, that in dimension~$d = 3$  the reduced Taylor--Hood element in~\cite{DieningStornTscherpel2022} has  significantly fewer degrees of freedom than the MINI element, see Tables~1 and~2 therein for a comparison.  
Here, we shall focus on the lowest-order MINI element in dimensions $d \in \{2,3\}$  and we address the higher-order cases in Remark~\ref{rmk:MINI-highorder} below. 

The pressure space $Q_h$  in~\eqref{def:Qh} is continuous with~$\mathcal{Q}=C(\overline{\Omega})$ and $Q(T) = \mathcal{P}_1(T)$. 
In the homogeneous setting the velocity finite element space~$X_h$ is given by \eqref{def:Xh} with local space~$X(T)= \mathcal{P}_1(T)^d + B(T)^d$, where $B(T)$ denotes the local element bubble space defined in~\eqref{def:local_element_bubblespace}. 
The standard construction of the Fortin operator in the homogeneous case is based on element bubble functions, see~\cite[Sec.~3]{GiraultLions2001} and \cite[A.1]{BelenkiBerselliDieningEtAl2012}. 
Due to the continuity of the pressure functions one may integrate by parts in the divergence preservation~\ref{prop:div-pres}. 
The correction operator is therefore based on this version. 

However, without zero traces the corresponding boundary terms do not vanish, and an enrichment with facet bubble functions~$B_n(T)$, as defined in~\eqref{def:local_normal_face_bubblespace}, is instrumental. 
Hence, we consider the modified element with 
	\begin{align}\label{def:modMINI}
	X(T)=\mathcal{P}_1(T)^d + B(T)^d + B_n(T) \qquad\text{for any } T\in \tria_h.
\end{align}

\begin{remark}\label{rmk:mod-MINI-1}
	The global modified MINI element spaces are
	$$Q_h=\mathcal{L}_1^1(\tria_h)
	\qquad \text{and}
	\qquad X_h=\mathcal{L}_1^1(\tria_h)^d + B(\tria_h)^d + B_n(\tria_h),$$ with bubble spaces $B(\tria_h)$ and $B_n(\tria_h)$ defined in~\eqref{def:global_el_bubblespace} and \eqref{def:global_normal_face_bubblespace}, respectively.
 Assumption~\ref{ass:FEM} is satisfied with $k = 1$.  
 
 Note the similarity with the modification of the lowest-order Crouzeix--Raviart element for~$d = 3$ and~$k= 2$  in Remark~\ref{rmk:confCRd3k2}.
\end{remark}

The Fortin operator~$\Pi_h$ is constructed as in~\eqref{eq:constr_inhom} 
with $C_h^0$ in Lemma~\ref{lem:Ch0} from the Bernardi--Raugel element, which works since $B_n(\tria_h) \subset X_h$. We choose~$C_h^1$ as the Crouzeix--Raviart operator in Lemma~\ref{lem:Ch1_CR} for $k = 2$. 
The latter is possible, since one has $B_{0,d}(T) = B(T)$, see~\eqref{def:Bk-2d}.  
	
\begin{proposition}[modified MINI element]\label{lem:Fortin_MINI}
	For the modified MINI element  (Remark~\ref{rmk:mod-MINI-1}) the operator~$\Pi_h\colon W^{1,1}(\Omega)^d \to X_h$ as in \eqref{eq:constr_inhom} with~$I_h$ as in Lemma~\ref{lem:SZinterp} mapping to~$X_h$, 
	$C_h^0$ as in Lemma~\ref{lem:Ch0}, and $C_h^1$ as in Lemma~\ref{lem:Ch1_CR} for $k = 2$ satisfies  properties~\ref{prop:lin-proj}--\ref{prop:local-stab-approx} with~$s =  k = 1$.	
	By suitable choice of~$I_h$ also~\ref{prop:trace-pres}$_{\Gamma}$ can be satisfied, see Remark~\ref{rmk:properties_Fortin}~\ref{itm:rmk_part-bd}.
\end{proposition}
\begin{proof}
Due to the enrichment, both operators~$C_h^0$ and~$C_h^1$ map to~$X_h$, see Remark~\ref{rmk:mod-MINI-1}.  
Verifying the assumptions of Proposition~\ref{prop:Chj_operators} relies on Lemma~\ref{lem:Ch0} and on Lemma~\ref{lem:Ch1_CR}. 
Finally, the divergence preservation follows from the proof of~Proposition~\ref{prop:Fortin_CR} in combination with the fact that the pressure space is contained in the lowest-order pressure space of the conforming Crouzeix--Raviart element, as $\mathcal{L}^1_1(\tria_h) \subset \mathcal{L}^0_1(\tria_h)$. 
\end{proof}

Due to the continuity of the pressure functions, the interior facet bubble functions can be avoided. 
To ensure the trace-preservation properties this comes at the price of using higher-order boundary facet bubble functions.

\begin{remark}[alternative modified MINI element]\label{rmk:MINI-alternative}

We present a modification of the original MINI element, where only boundary facet bubbles are added. 
The construction of the Fortin operator for the MINI element relies on integration by parts due to the continuity of the pressure functions. 
For functions $v \in W^{1,1}(\Omega)^d$ without trace conditions~\ref{prop:div-pres} is equivalent to 
	\begin{align}
0 = 	\int_\Omega \divergence (v-\Pi_h v) q_h \dx &=-\int_\Omega (v-\Pi_h v) \nabla q_h \dx + \int_{\partial\Omega} (v-\Pi_h v) \cdot n ~q_h \dsig \label{eq:div-pres_IBP_cont}
\end{align}
for any $q_h\in\overline{Q}_h$. 
To ensure this, we use a correction operator~$C_h^1$ based on element bubble functions to make the first term vanish, and we construct $C_h^0$ based on boundary facet bubble functions to make the trace term vanish. 
Note that since $q_h\big|_{T} \in \mathcal{P}_1(T)$, it does not suffice to use $C_h^0$ from Lemma~\ref{lem:Ch0}. 

This leads to the following modification of the velocity space of the MINI element
     \begin{align*}
     	X_h = \mathcal{L}_1^1(\tria_h)^d + B(\tria_h)^d + \BFp(\tria_h),
     \end{align*}
  with facet bubble function space $\BFp(\tria_h)$ defined in~\eqref{def:global_face_bubblespace_d+1}.

Here, for each boundary facet $f \in \facets^\partial{(\tria_h)}$ let $b_f^{(i)}$ for $i\in\{0,...,d-1\}$ be the associated boundary facet bubble functions with~$\support(b_f^{(i)}) \subset \omega_h(f)$, defined in~\eqref{def:face_bubble_higherdeg}. 
We construct the local operators~$C_{h,f}^0 \colon W^{1,1}(\Omega)^d \to \linearspan\{b_f^{(i)} n_f \colon i \in \{0, \ldots, d-1\} \}$ to satisfy 
\begin{align}
	\int_f (C_{h,f}^0v) \cdot n_f~q_h\dsig &= \int_f v\cdot n_f~q_h\dsig \qquad\text{for all }q_h\in\overline{Q}_h. \label{eq:constrMINI-3}
\end{align}
This is possible by solving a system of linear equations of size $d \times d$.  
Then, the global operator~$C_h^0 \colon W^{1,1}(\Omega)^d \to \BFp(\tria_h)$ is given by
\begin{align}
C_h^0 v \coloneqq \sum_{f \in \facets^\partial(\tria_h) } C_{h,f}^0 v \qquad \text{ for any } v \in W^{1,1}(\Omega)^d,
\end{align} 
and satisfies 
\begin{align}
	\int_{\partial \Omega} (C_{h}^0 v) \cdot n~q_h\dsig &= \int_{\partial \Omega} v\cdot n~q_h\dsig \qquad\text{for all }q_h\in\overline{Q}_h. \label{eq:constrMINI-4}	
\end{align}
For~$C_h^1$ we use the operator used in the homogeneous case~\cite{GiraultLions2001,BelenkiBerselliDieningEtAl2012}, mapping $W^{1,1}(\Omega)^d\to B(\tria_h)^d$ satisfying 
\begin{align}
	\int_T C_{h}^1 v \dx &= \int_T v \dx, \qquad\label{eq:constrMINI-2}
\end{align} 
for any~$T \in \tria_h$ and any~$v\in W^{1,1}(\Omega)^d$. 
This handles the first term in~\eqref{eq:div-pres_IBP_cont}, since $\nabla q_h$ is piecewise constant.   
One can show all properties in Proposition~\ref{prop:Chj_operators} with $s(0) = s(1) = 0$ (see, e.g., \cite[A.1]{BelenkiBerselliDieningEtAl2012}). 
Hence, for the operator $\Pi_h$ as in \eqref{eq:constr_inhom} the  properties~\ref{prop:lin-proj}--\ref{prop:local-stab-approx} with~$k = s = 1$  are available, and also~\ref{prop:trace-pres}$_{\Gamma}$,~see Remark~\ref{rmk:properties_Fortin}~\ref{itm:rmk_part-bd}.   	

The facet bubble functions  added to the velocity space are of one degree higher than in the approach presented in Remark~\ref{rmk:mod-MINI-1}. 
However, only the ones corresponding to boundary facets are needed. 
In this sense, this approach based on boundary facet bubble enrichment may still be less computationally costly in terms of degrees of freedom, but more challenging to implement. 

\end{remark}

\begin{remark}[higher-order MINI element]\label{rmk:MINI-highorder}
Let us briefly address the higher-order generalisation of the MINI element \cite[Sec.~8.4.2,~8.7.1]{Boffi2013}, even though it may be of lesser relevance in practice. 
For $k\geq1$ the spaces~$X_h$ and~$Q_h$ are defined in \eqref{def:Xh} and \eqref{def:Qh}, respectively, with local spaces  
	\begin{align}\label{def:genMINI}
		Q(T)\coloneqq \mathcal{P}_k(T)\qquad \text{and} \qquad X(T)=\mathcal{P}_k(T)^d + B_{k-1,d}(T)^d,
	\end{align} 
	see \eqref{def:Bk-2d} for a definition of the latter bubble function space. 
	Assumption~\ref{ass:FEM} is satisfied with degree~$k \in \mathbb{N}$.  
	Note that the pressure space of the  Crouzeix--Raviart element of order~$k+1$ contains the generalised MINI pressure space: $\mathcal{L}_k^1(\tria_h)\subset \mathcal{L}_k^0(\tria_h)$. 
	Hence, the Fortin operator for the MINI element of order~$k 
	\geq d$ can be constructed analogously to that of the  generalised conforming Crouzeix--Raviart element of order~$k+1$. 
	Recall that~$k \geq d$ ensures that~$B_n(\tria_h)\subset \mathcal{L}^1_{d}(\tria_h)^d \subset X_h$ for the MINI element, which allows us to employ the operator~$C_h^0$ from Lemma~\ref{lem:Ch0}. 	
	For~$k<d$ the operator~$C_h^0$ can still be used if the local velocity space is enriched with~$B_n(T)$ as  
	\begin{align*}
		X(T)=\mathcal{P}_k(T)^d + B_{k-1,d}(T)^d + B_n(T).
	\end{align*}
	Note the similarity with the enriched conforming Crouzeix--Raviart element in Remark~\ref{rmk:confCRd3k2}. 
	The Fortin operator for the generalised (enriched) MINI element defined with~$I_h$ as in Lemma~\ref{lem:SZinterp} for the (possibly enriched) velocity space, with~$C_h^0$ as in Lemma~\ref{lem:Ch0} and with~$C_h^1$ as in Lemma~\ref{lem:Ch1_CR} satisfies all desired properties with~$s = 1$, see Proposition~\ref{prop:Fortin_CR}.
\end{remark}

\subsubsection{Generalised Taylor--Hood element}\label{subsec:TH}
Taylor and Hood laid the foundation for the widely used Taylor--Hood element~\cite{Taylor1973}. 
The lowest-order Taylor--Hood element was subsequently analysed, e.g., in~\cite{BercovierPironneau1979,Verfuerth1984,Stenberg1984,Stenberg1987}. 
Its higher-order generalisation for~$d \in \{2,3\}$ and~$k\geq 2$ 	is defined by the local spaces $X(T) = \mathcal{P}_k(T)^d$ and $Q(T)=\mathcal{P}_{k-1}(T)$ with continuous pressure functions~$\mathcal{Q}=C(\overline{\Omega})$. 
Then, the global spaces $(X_h,Q_h)$ as defined in~\eqref{def:Xh} and~\eqref{def:Qh} are 
\begin{align}\label{def:TH-global}
	Q_h = \mathcal{L}_{k-1}^1(\tria_h) \qquad \text{and}\qquad X_h = \mathcal{L}_k^1(\tria_h)^d. 
\end{align}
They satisfy Assumption~\ref{ass:FEM} with $ k \geq 2$, and inf-sup stability in the higher-order cases was established in~\cite{BrezziFalk1991,Boffi1994,Boffi1997}, see also~\cite[Sec.~8.8.2]{Boffi2013}. 

\begin{assumption}\label{ass:corner-triangles}
	Each $d$-simplex $T\in\tria_h$ contains at least one interior vertex $z\in\vertices^\circ(\tria_h)$, i.e., if~$\gamma_T$ denotes the number of interior vertices of $T \in \tria_h$, then one has $\gamma_T \geq 1$ for all $T \in \tria_h$.    
\end{assumption}

Under this condition inf-sup stability holds for~$k \geq 2$ and~$d \in \{2,3\}$.  
For $d=2$ and $k\in\{2,3\}$ a local Fortin operator was constructed in~\cite{Falk2008}. 
While a strategy to construct a local Fortin operator was introduced in~\cite{Girault2003} for~$k\geq d$, only recently the lowest-order case~$k=2$ for $d\geq 3$ was established in~\cite{DieningStornTscherpel2022}, see also~\cite{DieningStornTscherpel2021,Scott2021}.
Note that~\cite{Girault2003,DieningStornTscherpel2022} treat homogeneous Dirichlet boundary conditions, whereas~\cite{Scott2021} addresses zero normal trace conditions. 
In the following, we construct a local Fortin operator for the generalised Taylor--Hood element~\eqref{def:TH-global} for $k  \geq d$  satisfying~\ref{prop:lin-proj}--\ref{prop:local-stab-approx} under Assumption~\ref{ass:corner-triangles}. 
This extends the constructions in~\cite{Girault2003,DieningStornTscherpel2022} and makes the one  in~\cite{Girault2003} more explicit. 
The remaining case~$k = 2$ and~$d = 3$ is addressed in Remark~\ref{rmk:THd3k2} below. 

In the construction of~$\Pi_h$ in~\eqref{eq:constr_inhom} we choose~$C_h^0$ as in Lemma~\ref{lem:Ch0} mapping to~$B_n(\tria_h) \subset X_h$, since $k \geq d$. 
Let us now address the construction of~$C_h^1\colon W^{1,1}(\Omega)^d\to X_{h,0}$ in terms of certain local operators~$C_{h,z}^1\colon W^{1,1}(\Omega)^d\to \mathcal{L}_{k,0}^1(\ohz)^d$ for interior vertices~$z\in{\vertices^\circ(\tria_h)}$. 
These operators are supposed to satisfy for any $v \in W^{1,1}(\Omega)^d$  
	\begin{align}
	\int_{\Ohz} (C_{h,z}^1 v) \cdot \nabla q_h \dx 
	&=
	- \sum_{T\in\ohz} \frac{1}{\gamma_T} \int_T \divergence v~ q_h \dx \qquad \text{for any } q_h \in \mathcal{L}_{k-1}^1(\ohz)\cap L_0^2(\Ohz),
	\label{eq:constr_TH-3}
\end{align}
where~$\gamma_T \in \mathbb{N}$ is the number of interior vertices of a simplex~$T\in\tria_h$, see Assumption~\ref{ass:corner-triangles}. 
Then, the local operators are extended by zero and added to obtain 
\begin{align}\label{def:Ch1_TH}
	C_h^1 v \coloneqq \sum_{z \in \vertices^\circ(\tria_h)} C_{h,z}^1 v \qquad \text{ for } v \in W^{1,1}(\Omega)^d. 
\end{align}
Before addressing existence of the local operators $C_{h,z}^1$, let us motivate their construction by presenting the resulting divergence-preserving properties of~$C_{h}^1$. 

\begin{lemma}[divergence preservation of $C_h^1$]\label{lem:TH-div-Ch1}
If $C_h^1 \colon W^{1,1}(\Omega)^d \to X_{h,0}$ is constructed as in \eqref{def:Ch1_TH} with~$C_{h,z}^1\colon W^{1,1}(\Omega)^d\to \mathcal{L}_{k,0}^1(\ohz)^d$ satisfying
\eqref{eq:constr_TH-3} for any $z \in {\vertices^\circ(\tria_h)}$, then one has 
\begin{align*}
	\int_{\Omega} \divergence (v - C_{h}^1 v)~ q_h \dx = 0 \qquad \text{ for any } q_h \in \overline{Q}_h, 
\end{align*}   
and any $v \in W^{1,1}(\Omega)^d$ satisfying $\int_{f} v \cdot n_f \, \mathrm{d} \sigma(x) = 0$ for any $f \in \facets{(\tria_h)}.$   
\end{lemma} 
\begin{proof}
	Fix an arbitrary function $v \in W^{1,1}(\Omega)^d$ satisfying~$\int_{f} v \cdot n_f \, \mathrm{d} \sigma(x) = 0$ for any facet~$f \in \facets(\tria_h)$. Furthermore, let $q_h \in \overline{Q}_h$ be arbitrary. 
	By writing for fixed $z \in \vertices^\circ(\tria_h)$ the pressure function as  $$ q_h = q_h - \langle q_h \rangle_{\Ohz} + \langle q_h \rangle_{\Ohz}$$ 
	and using  $\int_{\partial T } v \cdot n \dsig  = 0$ for any $T\in \tria_h$, and~\eqref{eq:constr_TH-3}, we obtain  
 	\begin{align}\label{eq:TH-div-1}
 	\int_{\Ohz} (C_{h,z}^1 v) \cdot \nabla q_h \dx 
 =
 - \sum_{T\in\ohz} \frac{1}{\gamma_T} \int_T \divergence v~ q_h \dx.
 	\end{align} 
Since the global operator~$C_{h}^1$ maps to~$W^{1,1}_0(\Omega)^d$, integrating by parts,  using the definition of~$C_h^1$ in~\eqref{def:Ch1_TH} and employing ~\eqref{eq:TH-div-1} yields 
\begin{align*}
	\int_{\Omega} \divergence (C_{h}^1 v)~  q_h \dx 
	& = 
	-	\int_{\Omega}  (C_{h}^1 v)  \cdot \nabla q_h \dx  
	= - \sum_{z \in \vertices^\circ(\tria_h)} \int_{\Ohz} C_{h,z}^1 v \cdot \nabla q_h \dx \\
	& =  \sum_{z \in \vertices^\circ(\tria_h)}  \sum_{T\in\ohz} \frac{1}{\gamma_T} \int_T \divergence v~ q_h \dx
	 = \int_{\Omega} \divergence v~ q_h \dx,
\end{align*}	
where we have used that each $T \in \tria_h$ occurs $\gamma_T$ times in the sum. 
This proves the claim.   
\end{proof}

Next, we address well-posedness of the local operators $C_{h,z}^1$ satisfying~\eqref{eq:constr_TH-3}. 
Let~$z \in {\vertices^\circ(\tria_h)}$ be arbitrary but fixed. 
Let~$(q_i)_{i\in\{1,\ldots,\delta\}}$ be a basis of
\begin{align}\label{eq:TH_di}
	\overline{Q}_h(\Ohz) \coloneqq \mathcal{L}_{k-1}^1(\ohz)\cap L_0^2(\Ohz) \qquad \text{ for}\quad \delta \coloneqq \dim(\overline{Q}_h(\Ohz)).
\end{align}
We choose this basis by transformation from a fixed basis on a reference vertex patch, of which there are only finitely many combinatorial types. 
Note that the pressure basis functions~$q_i$ for $i\in\{1,\dots,\delta\}$ are non-constant, i.e., $\nabla q_i \not\equiv 0$. 
Similarly to several previous results in this direction (see Remark~\ref{rmk:tang-edge} below)
we use tangential edge bubble functions for all edges meeting at vertex $z$. 
Specifically, for~$e\in\Ehz$, and $\tau_e$ one of its unit tangential vectors, we use~$b_e^\tau= b_e \tau_e$  defined in~\eqref{def:tan_edge_bubble} and make the ansatz
\begin{align}\label{eq:constr_TH}
	C_{h,z}^1 v = \sum_{i=1}^{\delta} \alpha_i \sum_{e\in\Ehz} b_e (\nabla q_i \cdot \tau_e)\tau_e, 
\end{align}
with some coefficient vector $\alpha\in\R^{\delta}$.  
Since $\support(b_e) \subset\overline{\Omega}_z $ for any $e \in \Ehz$ and $b_e \in \mathcal{L}^1_{2}(\omega_h(e))$ and $\nabla q_i \in \mathcal{L}^0_{k-2}(\ohz)^d$,  we have  that $C_{h,z}^1$ indeed maps to~$\mathcal{L}^1_{k,0}(\ohz)^d$.  
Note that continuity is ensured, since only the tangential derivative of $q_i$ occurs, which exists and is continuous since $q_i$ is continuous.     
Determining~$\alpha \in \setR^\delta$ such that $C_{h,z}^1 v$ satisfies~\eqref{eq:constr_TH-3} is equivalent to solving the linear system~$M\alpha=F$ with~$M \in \setR^{\delta \times \delta}$ and~$F \in \setR^{\delta}$ given by 
\begin{align}\label{eq:TH_lin-system}
	M_{i,j}
	\coloneqq \dashint_{\Ohz} \sum_{e\in\Ehz} b_e (\nabla q_i\cdot\tau_e)(\nabla q_j\cdot\tau_e) \dx,
	\quad F_j
	\coloneqq - \frac{1}{\abs{\Ohz}} \sum_{T\in\ohz} \frac{1}{\gamma_T} \int_T \divergence v ~q_j \dx,
\end{align}
for $i,j\in\{1,\dots,\delta \}$. 
For $C_{h,z}^1$ to be well-defined, one has to show that this system has a unique solution. 

 The following is analogous to Lemma~\ref{lem:CR-matrix} for the Crouzeix--Raviart element. 
We use the notation~$h_z \coloneqq \tfrac{1}{\abs{\ohz}} \sum_{T \in \ohz} h_T$, which is equivalent to~$h_T$ for any~$T \in \ohz$ due to shape regularity.

\begin{lemma}\label{lem:TH-M-1}
	For a vertex~$z \in {\vertices^\circ(\tria_h)}$ the corresponding matrix~$M\in\R^{\delta \times \delta }$ defined in \eqref{eq:TH_lin-system} is symmetric and positive definite. 
	Furthermore, there is a constant~$C>0$ (depending only on~$k,d$ and the shape-regularity constant) such that 
	\begin{align*}
		\norm{M^{-1}}_2 \leq C h_z^{2},
	\end{align*}
	for any~$z \in {\vertices^\circ(\tria_h)}$, and uniformly in $h$.   
\end{lemma}
\begin{proof}
	By definition the matrix $M$ is symmetric. 
	Let us show that it is positive definite. For this purpose, let~$\alpha\in\R^{\delta}\backslash\{0\}$ be arbitrary and fixed, and set~$q_h\coloneqq \sum_{i=1}^{\delta} \alpha_i q_i\in \overline{Q}_h(\Ohz)$. 
	Since~$b_e\geq 0$ in $\Ohz$ we have that $M$ is positive semidefinite, as  
		\begin{align}
		\abs{\Ohz}	\alpha^\top M \alpha 
			=  \sum_{e\in\Ehz}\int_{\Ohz} b_e \abs{\nabla q_h\cdot\tau_e}^2 \dx
			=  \sum_{e\in\Ehz}  \sum_{T\in\omega_h(e)}\int_{T} b_e \abs{\nabla q_h\cdot\tau_e}^2 \dx
			\geq 0. \label{eq:TH_M_posdef}
		\end{align}

		To show that~$M$ is positive definite, assume that~$\alpha^\top M \alpha=0$ for some~$\alpha\in\R^{\delta}\backslash \{0\}$ with corresponding~$q_h$.  
		In this case, each of the terms in~\eqref{eq:TH_M_posdef} has to vanish. 
		With~$b_e>0$ on~$\Omega_h(e)$, it follows that for $e\in\Ehz$ and all $T\in\omega_h(e)$ one has for the polynomial $\nabla q_h\big|_{T}$  that
		\begin{align}\label{eq:TH_nablaqdottau}
			\nabla q_h\cdot \tau_e = 0 \quad \text{ on } T.
		\end{align}
		Fix one~$T\in\ohz$ and note that $\{\tau_e \colon e\in\Ehz,  e\subseteq T\}$ are linearly independent and span $\R^d$. 
		Hence, \eqref{eq:TH_nablaqdottau} shows that $\nabla q_h \equiv 0$ on $T$, i.e., $q_h$ is constant on $T$. 
		Since  $q_h \in \overline{Q}_h(\Ohz)\subseteq C(\overline{\Omega}_z)$ it follows that $q_h$ is constant on~$\Ohz$, and with~$q_h \in L^2_0(\Ohz)$ it follows that~$q_h\equiv 0$. 
		This shows that $M$ is indeed positive definite.		
		
		To obtain an estimate on $\norm{M^{-1}}_2$ we need to estimate $\lambda_{\min}$, the lowest eigenvalue of $M$, below.    
		Due to the shape regularity of the sequence of triangulations, the number $m_z$ of simplices in $\ohz$ is bounded by some $\mu \in \mathbb{N}$
		\begin{align*}
			m_z \leq \mu,
		\end{align*}
		uniformly in $z \in \vertices^\circ(\tria_h)$ and uniformly in $h$. 
		Furthermore, there are only finitely many combinatorial types~$t \in \{1, \ldots, \theta\}$  of vertex patches, determined by the adjacency structure of the boundary vertices of $\ohz$. 
		Here~$\theta \in \mathbb{N}$ depends on the shape-regularity constant and the dimension only. 
		For each combinatorial type~$t \in \{1, \ldots, \theta\}$ we consider a respective reference patch $ \hat{\omega}_{t}$ consisting of $m = m(t) \leq \mu$ simplices $\hat{T}_1, \ldots, \hat{T}_{m}$ such that each simplex has~$0$ as vertex and~$d$ vertices on $\partial B_1(0)$. 
		Since we consider interior vertices~$z \in \vertices^\circ(\tria_h)$, the reference patches contain a ball $B_r(0)$ for some $r>0$.  We refer to Figure~\ref{fig:vertexpatch-reference} for a visualisation in the simple case of dimension~$d = 2$.      
		
		Let~$\hat{\Omega}_t$ denote the domain covered by the reference patch $\hat{\omega}_t$, i.e., $\hat{\Omega}_t=\mathrm{int}(\cup_{i=1}^{m} \hat{T}_i)$. 
		Let $\hat{\edges}_t$ denote all interior edges in the reference patch $\hat{\omega}_t$. 
		Then, for each $z \in {\vertices^\circ(\tria_h)}$, there is a reference patch $\hat{\omega}_t$ for a $t \in \{1, \ldots, \theta\}$ such that there is a continuous, piecewise affine bijection $\Phi_{z}\colon \hat{\Omega}_t \to \Ohz$, mapping vertices to vertices. 
		Specifically, there is a labelling $T_1, \ldots T_{m_z}$ of the simplices in $\ohz$ such that $\Phi_z(\hat T_j) = T_j$ for $j \in \{1, \ldots, {m_z}\}$  and for each 
		$T \in \ohz$ with corresponding~$\hat T \in \hat \omega_t$  there is a regular matrix $A_T\in\R^{d\times d}$ and $a_T\in\R^d$ such that $\Phi_z\big|_{\hat{T}}
		\colon\hat{T}\to T$ maps~$\hat{x}\mapsto A_T \hat{x} + a_T$ for $\hat x \in \hat{T}$.  
		
		We consider the transformed  basis~$\hat q_i$, and the edge bubble functions~$\hat{b}_{\hat{e}}(\hat{x})=b_e(x)$ for $x=\Phi_z(\hat{x})$, for~$ \hat e \in \hat{\edges}_t$. 
		The unit tangent vectors~$\hat{\tau}_{\hat{e}}$ in~$\hat{\Omega}_t$ are related to~$\tau_e$ through
		\begin{align*}
			\tau_e = \frac{A_T \hat{\tau}_{\hat{e}}}{|A_T \hat{\tau}_{\hat{e}}|}.
		\end{align*}
		With the transformation of gradient as~$\nabla q_h(x) = A_T^{-\top} \hat{\nabla}\hat{q}_h(\hat{x})$, see~\eqref{eq:scaling}, we have 
		\begin{align*}
			\nabla q_h(x)\cdot \tau_e = \frac{A_T^{-\top} \hat{\nabla}\hat{q}_h(\hat{x})\cdot A_T \hat{\tau}_{\hat{e}}}{|A_T \hat{\tau}_{\hat{e}}|} =  \frac{\hat{\nabla}\hat{q}_h(\hat{x})\cdot\hat{\tau}_{\hat{e}}}{|A_T \hat{\tau}_{\hat{e}}|} \eqsim h_z^{-1} \hat{\nabla}\hat{q}_h(\hat{x}) \cdot \hat{\tau}_{\hat{e}},
		\end{align*}
		since~$|A_T \hat{\tau}_{\hat{e}}|$ scales with $h_T$ and the shape regularity of the mesh yields $h_T\eqsim h_z$ for all $T\in\ohz$. 
		Thus, for  the reference matrix $\hat{M}\in\R^{\delta \times \delta }$ with entries
		\begin{align*}
			\hat{M}_{i,j} \coloneqq \sum_{\hat{e}\in\hat{\edges}_t} \dashint_{\hat{\Omega}_t}  \hat{b}_{\hat{e}} (\hat{\nabla}\hat{q}_i\cdot \hat{\tau}_{\hat{e}}) (\hat{\nabla}\hat{q}_j\cdot \hat{\tau}_{\hat{e}}) \dxhat\qquad \text{for } i,j \in\{1,...,\delta\},
		\end{align*}
		with scaling arguments we obtain 
		\begin{align}
			\alpha^\top M \alpha 
			= 
		\sum_{e\in\Ehz}
		\dashint_{\Ohz} b_e \abs{\nabla q_h\cdot\tau_e}^2 \dx	
		\eqsim 
		h_{z}^{-2} 
			\sum_{\hat e\in\hat \edges_t}\dashint_{\hat \Omega_t} \hat{ b}_{\hat{e}} \abs{\hat \nabla \hat{q}_h\cdot \hat \tau_{\hat{e}}}^2 \, \mathrm{d} \hat x
			 =
			   h_{z}^{-2}
					\alpha^\top \hat M \alpha,  	
		\end{align}
		and consequently $\hat{M}$ is also positive definite. 
		In particular, its smallest eigenvalue depends only on $k$, $d$ and the shape-regularity constant. 
		With 
		\begin{align*}
			\alpha^\top M \alpha  \geq c h_z^{-2} \abs{\alpha}^2	
		\end{align*}
		this proves the claim with constant $c$ depending on $k,d$ and the shape-regularity constant.       
		
\end{proof}	

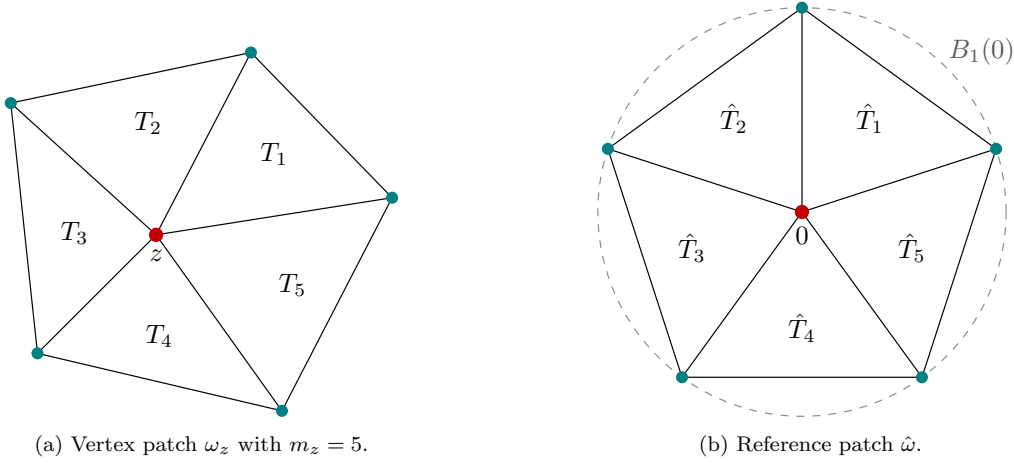
\begin{figure}[ht]
	\centering
	
	\tikzset{
		mesh edge/.style={
			black,
			line width=0.45pt
		},
		mesh vertex/.style={
			circle,
			fill=teal,
			inner sep=1.6pt
		},
		central vertex/.style={
			circle,
			fill=red!75!black,
			inner sep=1.9pt
		},
		unit circle/.style={
			black!45,
			dashed,
			line width=0.45pt
		}
	}
	
	\begin{subfigure}[t]{0.46\textwidth}
		\centering
		
		\begin{tikzpicture}[scale=0.9]
			
			\coordinate (z) at (-0.3,-0.1);
			
			\coordinate (A) at (8:3.20);
			\coordinate (B) at (67:2.80);
			\coordinate (C) at (143:3.05);
			\coordinate (D) at (222:2.75);
			\coordinate (E) at (300:3.10);
			
			\draw[mesh edge]
			(A)--(B)--(C)--(D)--(E)--cycle;
			
			\draw[mesh edge]
			(z)--(A)
			(z)--(B)
			(z)--(C)
			(z)--(D)
			(z)--(E);
			
			\node[central vertex] at (z) {};
			
			\node[mesh vertex] at (A) {};
			\node[mesh vertex] at (B) {};
			\node[mesh vertex] at (C) {};
			\node[mesh vertex] at (D) {};
			\node[mesh vertex] at (E) {};
			
			\node[below=2pt] at (z) {$z$};
			
			\node at (37:1.8) {$T_1$};
			\node at (105:1.6) {$T_2$};
			\node at (183:1.5) {$T_3$};
			\node at (261:1.6) {$T_4$};
			\node at (334:1.9) {$T_5$};
			
		\end{tikzpicture}
		
		\caption{
			Vertex patch $\ohz$ with
			$m_z =5$.
		}
	\end{subfigure}
	\hfill
	\begin{subfigure}[t]{0.46\textwidth}
		\centering
		
		\begin{tikzpicture}[scale=0.9]
			
			\coordinate (zhat) at (0,0);
		
			\coordinate (Ahat) at (90:3);
			\coordinate (Bhat) at (18:3);
			\coordinate (Chat) at (-54:3);
			\coordinate (Dhat) at (-126:3);
			\coordinate (Ehat) at (162:3);
			
			\draw[unit circle] (zhat) circle (3);
			
			\draw[mesh edge]
			(Ahat)--(Bhat)--(Chat)--(Dhat)--(Ehat)--cycle;
			
			\draw[mesh edge]
			(zhat)--(Ahat)
			(zhat)--(Bhat)
			(zhat)--(Chat)
			(zhat)--(Dhat)
			(zhat)--(Ehat);
			
			\node[central vertex] at (zhat) {};
			
			\node[mesh vertex] at (Ahat) {};
			\node[mesh vertex] at (Bhat) {};
			\node[mesh vertex] at (Chat) {};
			\node[mesh vertex] at (Dhat) {};
			\node[mesh vertex] at (Ehat) {};
			
			\node[below =2pt] at (zhat) {$0$};

			\node at (54:1.70) {$\hat T_1$};
			\node at (-18:1.70) {$\hat T_5$};
			\node at (-90:1.70) {$\hat T_4$};
			\node at (-162:1.70) {$\hat T_3$};
			\node at (126:1.70) {$\hat T_2$};

			\node[
			black!65,
			above right=-1pt
			]
			at (2.05,2.05) {$B_1(0)$};
			
		\end{tikzpicture}
		
		\caption{
			Reference patch $\hat{\omega}$.
		}
		\label{fig:reference-macroelement-five}
	\end{subfigure}
	
	\caption{
		Vertex patch $\ohz$ of vertex $z$ with $m_z=5$ simplices and its reference vertex patch.}
	\label{fig:vertexpatch-reference}
\end{figure}

\begin{lemma}[properties of $C_h^1$]\label{lem:Ch1_TH-stab}
		The linear operator~$C_h^1\colon  W^{1,1}(\Omega)^d\to X_{h,0}$, defined in \eqref{def:Ch1_TH}, is well-defined and maps to~$\mathcal{L}^1_{k,0}(\tria_h)^d$. 
For any~$v \in W^{1,1}(\Omega)^d$ with $\int_{f} v \cdot n_f \, \mathrm{d} \sigma(x) = 0$ for any~$f \in \facets(\tria_h)$ it satisfies that     
	\begin{align}\label{eq:TH-div-Ch1}
		\int_{\Omega} \divergence (v - C_{h}^1 v)~ q_h \dx = 0 \qquad \text{ for any } q_h \in \overline{Q}_h. 
	\end{align}   
	Further, it is locally $W^{1,1}$-stable, that is 
	\begin{align*}
		\norm{C_h^1v}_{L^1(T)} \lesssim	h_T\norm{\nabla v}_{L^1(\Omega_h(T))} \quad \text{for all } v\in W^{1,1}(\Omega)^d,~\text{for all }T\in\tria_h,
	\end{align*}
uniformly in $h$. 
\end{lemma}
\begin{proof}
	By Lemma~\ref{lem:TH-M-1} for each $z \in {\vertices^\circ(\tria_h)}$ the matrix $M$ as in \eqref{eq:TH_lin-system} is positive definite, and hence~$C_{h,z}^1 \colon W^{1,1}(\Omega)^d \to \mathcal{L}_{k,0}^1(\ohz)^d$ satisfying 
~\eqref{eq:constr_TH-3} is well-defined. 
The divergence-preservation property is proved in Lemma~\ref{lem:TH-div-Ch1}.

To show the local stability let~$T \in \tria_h$ be arbitrary but fixed and let~$v \in W^{1,1}(\Omega)^d$. 
Furthermore, for fixed $z \in {\vertices^\circ(\tria_h)}$ such that~$z \in T$ we consider~$C_{h,z}^1$ satisfying~\eqref{eq:constr_TH-3}. 
We apply~\eqref{eq:constr_TH} and~\eqref{eq:TH_lin-system}, together with the fact that $\dashint_{T} b_e \dx$ is a constant  depending only on the shape regularity and $d$, the scaling property  $\norm{{q_i}}_{L^\infty(T)} +  h_T \norm{\nabla q_i}_{L^\infty(T)} \lesssim 1$ for any~$i \in \{1, \ldots, \delta\} $, the equivalence of norms on $\setR^\delta$, and $M \alpha = F$ with the bound on $\norm{M^{-1}}_{2}$ in Lemma~\ref{lem:TH-M-1} and shape regularity to find that 
\begin{equation}\label{est:stab-Ch1z} 
 \begin{aligned}
 	\dashint_T |C_{h,z}^1 v|\dx
 	& =
 	\dashint_T \left| \sum_{i= 1}^{\delta} \alpha_i \sum_{e\in\Ehz} b_e(\nabla q_i\cdot \tau_e)\tau_e \right| \dx
 	\lesssim
 	\sum_{i=1}^{\delta }\abs{\alpha_i} \norm{\nabla q_i}_{L^\infty(T)}  \sum_{\substack{e\in\Ehz:  \\e \subset T}} \dashint_T b_e \dx \\
 	&\lesssim 
  h_T^{-1} 	\abs{\alpha} 
 	 \lesssim 
 	h_T^{-1} \norm{M^{-1}}_2 \abs{F} 
 	\lesssim 
  h_T^{-1}  	h_z^{2} \max_{j \in \{1, \ldots, \delta\}} \abs{F_j} \\
 	& \lesssim 
 	h_T  {\max_{j \in \{1, \ldots, \delta\}}}\norm{q_j}_{L^{\infty}(\Ohz)} \dashint_{\Ohz}\abs{\divergence v} \dx  
 	\lesssim h_T \dashint_{\Omega_h(T)} \abs{\nabla v}\dx.
 \end{aligned}
 \end{equation}
   Then, using the definition of the operator~$C_h^1$ in~\eqref{def:Ch1_TH},  the fact that the number of vertices per simplex is bounded, and the estimate~\eqref{est:stab-Ch1z} we find that  
	\begin{align*}
		\dashint_T \abs{C_h^1v} \dx  & 
		\leq 
	\sum_{z \in T \cap \vertices^\circ(\tria_h)}	\dashint_T  
		\abs{C_{h,z}^1 v}\dx 
		\lesssim h_T \dashint_{\Omega_h(T)} \abs{\nabla v}\dx,
	\end{align*}
	which proves the claim.
\end{proof}

Now we are in a position to collect the properties of the local Fortin operator for the Taylor--Hood element for $k \geq d$, see~\eqref{def:TH-global}. 

\begin{proposition}[{generalised} Taylor--Hood element]\label{prop:Fortin_TH}
	Let Assumption~\ref{ass:corner-triangles} be satisfied. 
	Then, for the Taylor--Hood element of order~$k \geq d$ the operator~$\Pi_h\colon W^{1,1}(\Omega)^d \to X_h$ defined by~\eqref{eq:constr_inhom} with~$I_h \colon W^{1,1}(\Omega)^d \to X_h$ as in Lemma~\ref{lem:SZinterp}, with the operator $C_h^0$ in Lemma~\ref{lem:Ch0} and with the operator $C_h^1$~in Lemma~\ref{lem:Ch1_TH-stab}, is a Fortin operator satisfying the properties \ref{prop:lin-proj}--\ref{prop:local-stab-approx} with~$s = 2$. 
	
	Choosing $I_h$ suitably,   also~\ref{prop:trace-pres}$_{\Gamma}$ can be satisfied, see Remark~\ref{rmk:properties_Fortin}~\ref{itm:rmk_part-bd}.
\end{proposition}
\begin{proof}
	We need to check the conditions of Proposition~\ref{prop:Chj_operators}. 
	Since~$k \geq d$ we have $B_{n}(\tria_h) \subset X_h$, and thus $C_h^0$ in Lemma~\ref{lem:Ch0} is available. 
	By Remark~\ref{rmk:Ch0} it remains to verify the properties of $C_h^1$ and 
	the divergence-preservation property in Proposition~\ref{prop:Chj_operators}~\ref{itm:lemma_div-pres}.  
	The former are satisfied thanks to Lemma~\ref{lem:Ch1_TH-stab}. 
	To show the latter, let~$q_h \in \overline Q_h$ and~$v \in W^{1,1}(\Omega)^d$ be arbitrary.   
	By the properties of~$C_h^0$  in Lemma~\ref{lem:Ch0} we have that 
	\begin{align*}
		\int_{f} (v - C_h^0 v) \cdot n_f  \dsig = 0 \qquad \text{ for any } f \in \facets(\tria_h).  
	\end{align*}
	Hence, we may apply~\eqref{eq:TH-div-Ch1} in Lemma~\ref{lem:Ch1_TH-stab} to $v - C_h^0 v$, which shows 
	\begin{align*}
		\int_{\Omega} \divergence ((I - C_{h}^1)(I - C_h^0)(v))~ q_h \dx = 0 \qquad \text{ for any } q_h \in \overline{Q}_h. 
	\end{align*}   
	This shows the claim for $C_h= C_h^0+C_h^1(I-C_h^0)$ in Proposition~\ref{prop:Chj_operators}~\ref{itm:lemma_div-pres} and finishes the proof.
\end{proof}

\begin{remark}\label{rmk:tang-edge}
	Tangential edge bubble functions have also been used in the construction  in~\cite{DieningStornTscherpel2022} for the case $k = 2$, $d \geq 2$, and  in~\cite{Scott2021}. 
Previously, they have been employed in dimension $d = k = 2$, e.g., in~\cite{BercovierPironneau1979}, \cite[Ch.~II, Sec.~4.2]{Girault1986},~\cite[Sec.~3.1]{Girault2003}, \cite{Falk2008}, \cite{MardalSchoeberlWinther2013}, 
	and for higher-order elements in~\cite{Boffi1997,Falk2008}, see also~\cite[eq.~(5.5.9)]{Bernardi2024}. 
	
	Similarly to the observation in~\cite{MardalSchoeberlWinther2013,DieningStornTscherpel2022}, the construction of our Fortin operator gives rise to a reduced Taylor--Hood element, which only uses $\mathcal{L}^1_1(\tria_h)^d$, $B_n(\tria_h)$ and tangential edge bubble functions.
	\end{remark}

	\begin{remark}[alternative correction operators for $k = 2$ ]\label{rmk:Ch1-DST}
		
		The construction in the lowest-order case $k = 2$ in~\cite{DieningStornTscherpel2021,Scott2021,DieningStornTscherpel2022} uses tangential edge bubble functions to obtain a correction operator $\tilde C_h^1 \colon L^{1}(\Omega)^d \to \mathcal{L}^1_2(\tria_h)^d \cap W^{1,1}_n(\Omega) = X_{h,n}$ which is locally $W^{1,1}$-stable and satisfies 
		\begin{align}\label{eq:div-pres-DST}
			\int_{\Omega} (\tilde C_h^1 v - v) \cdot \nabla q_h \dx = 0 \qquad \text{ for any } v \in L^1(\Omega)^d \text{ and any } q_h \in Q_h. 
		\end{align} 
		In particular, this operator satisfies $\tilde C_h^1(W^{1,1}_n(\Omega)) \subset X_{h,n}$, but it does not preserve zero traces, since it uses boundary tangential edge bubble functions. 
		However, since it maps to $X_{h,n}$ one may integrate by parts in~\eqref{eq:div-pres-DST}
		and for any~$v \in W^{1,1}_n(\Omega)$ it follows that 
		\begin{align}\label{eq:div-pres-DST-hom}
			\int_{\Omega} \divergence (\tilde C_h^1 v - v)  q_h \dx = 0 \qquad \text{ for any } q_h \in Q_h. 
		\end{align}
		The resulting Fortin operator~$\Pi_h$ as in~\eqref{eq:constr_inhom} with $C_h^1 \coloneqq \tilde C_h^1$ and $C_h^0 \equiv 0$ satisfies properties~\ref{prop:lin-proj}, \ref{prop:div-pres} for $v \in W^{1,1}_n(\Omega)$,
		\ref{prop:trace-pres}~\ref{itm:global_zeronormaltrace}, as well as ~\ref{prop:global-stab-approx} and \ref{prop:local-stab-approx} with $k = 2$ and $s = 2$. 
		
		To ensure that homogeneous traces are also preserved, in~\cite{DieningStornTscherpel2022} the operator $\tilde C_h^1$ was modified so that the use of boundary tangential edge bubble functions is avoided. 
		The resulting operator $C^1_h \colon L^1(\Omega)^d \to \mathcal{L}^1_{2,0}(\tria_h)^d = X_{h,0}$ is still locally $W^{1,1}$-stable and satisfies~\eqref{eq:div-pres-DST}.  
		By construction, it satisfies $C^1_h(W^{1,1}_0(\Omega)^d) \subset X_{h,0}$, and \eqref{eq:div-pres-DST-hom} still holds. 
		The resulting Fortin operator $\Pi_h$ as in \eqref{eq:constr_inhom} with~$C_h^0 \equiv 0$ satisfies properties~\ref{prop:lin-proj}, \ref{prop:div-pres} for $v \in W^{1,1}_n(\Omega)$,
		\ref{prop:trace-pres}~\ref{itm:global_zerotrace} and \ref{itm:global_zeronormaltrace}, as well as ~\ref{prop:global-stab-approx} and \ref{prop:local-stab-approx} with $k = 2$ and $s = 2$, see~\cite[Sec.~2.4]{DieningStornTscherpel2022}. 
		This modification is not needed when preservation of homogeneous traces is not important. 
		Note that this construction does not ensure~\ref{prop:trace-pres}~\ref{itm:global_meantrace}, nor \ref{prop:div-pres} for $v \in W^{1,1}(\Omega)$.
		
		In contrast, here we address trace preservation with an extra correction operator, which allows us to obtain all trace-preservation properties.  
		This comes at the price of an enrichment in the lowest-order case, as described in the following remark.  
	\end{remark}

\begin{remark}[$d = 3$ and $k = 2$]\label{rmk:THd3k2}
	The assumption $k\geq d$ ensures that the facet bubble functions needed for $C_h^0$ are contained in $X_h$, see Lemma~\ref{lem:Ch0}. 
	Clearly, this does not cover the important  case $d=3$ and $k=2$. 
	In this case one may enrich the velocity space, similarly to the MINI element in Section~\ref{subsec:modMINI}. There are two options: 
	\begin{enumerate} 
	\item \label{itm:THd3k2-opt1} 
	The first option (similar to Section~\ref{subsec:modMINI} for the modified MINI element) consists in choosing the velocity space 
	\begin{align*}
			X_h = \mathcal{L}_2^1(\tria_h)^3 + B_n(\tria_h),
		\end{align*}
	with~$B_n(\tria_h)$ defined in ~\eqref{def:global_normal_face_bubblespace}.  
	Then, $\Pi_h$ is constructed using $C_h^0$ in Lemma~\ref{lem:Ch0} and $C_h^1$ in Lemma~\ref{lem:Ch1_TH-stab} and the proof of its properties proceeds as before.  
	\item \label{itm:THd3k2-opt2} 
	Alternatively, one may enrich with boundary facet bubbles only, similarly to Remark~\ref{rmk:MINI-alternative} for the MINI element. 
	More precisely, the velocity space can be chosen as 
	\begin{align*}
	X_h = \mathcal{L}_2^1(\tria_h)^3 + \BFp(\tria_h),
\end{align*}
with~$\BFp(\tria_h) \subset \mathcal{L}^1_{4}(\tria_h)^3$ defined in~\eqref{def:global_face_bubblespace_d+1}.  
Then, $\Pi_h$ is constructed using $C_h^1$ from \cite{DieningStornTscherpel2022} as summarised in Remark~\ref{rmk:Ch1-DST}, and with $C_h^0$ as in Remark~\ref{rmk:MINI-alternative} for the MINI element. 
This is natural because for both elements the local pressure space is~$Q(T) = \mathcal{P}_1(T)$.  
The only property which is not immediate is the divergence preservation for $v \in W^{1,1}(\Omega)^3$, see Proposition~\ref{prop:Chj_operators}. 
For arbitrary $w \in W^{1,1}(\Omega)^3$ and $q_h \in \overline{Q}_h$,  using integration by parts, the fact that $C_h^1$ maps to $X_{h,0}$ and property~\eqref{eq:div-pres-DST} we find    
\begin{equation}\label{eq:TH-div-2}  
\begin{aligned}
	\int_{\Omega} \divergence(C_h^1(I-C_h^0) w)\, q_h \dx 
	& = -
		\int_{\Omega} (C_h^1(I-C_h^0) w) \cdot \nabla  q_h \dx 
	 \\
	& = -  	\int_{\Omega} ((I-C_h^0) w) \cdot \nabla  q_h \dx.  
\end{aligned}
\end{equation}
Then, for arbitrary~$v \in W^{1,1}(\Omega)^3$ we choose $w = v - I_h v$. 
Using the definition of $\Pi_h$ in~\eqref{eq:constr_inhom},~\eqref{eq:TH-div-2}, integrating by parts, and finally~\eqref{eq:constrMINI-4}, we obtain
\begin{align*}
\int_{\Omega} \divergence(v- \Pi_h v) q_h \dx 
	&= 	\int_{\Omega} \divergence((I-C_h^0) w)\, q_h \dx 
	- 	\int_{\Omega} \divergence(C_h^1(I-C_h^0) w)\, q_h \dx \\
&=	\int_{\Omega} \divergence((I-C_h^0) w)\, q_h \dx 
+ 	\int_{\Omega} ((I-C_h^0) w) \cdot \nabla  q_h \dx  \\
& = \int_{\partial \Omega} (w - C_h^0 w) \cdot n\, q_h \dsig = 0 \qquad \text{ for any } q_h \in \overline{Q}_h.
\end{align*}
	\end{enumerate}
\end{remark}	
	

\section{Abstract global Fortin operators}\label{sec:globalFortin}
In this section we present a strategy to obtain a Fortin operator~$\Pi_h^p$ with trace-preservation properties~\ref{prop:trace-pres}, provided that a Fortin operator for homogeneous traces $\Pi_{h,0}^{p}$ is available. 
This is of interest, since the local constructions in~Section~\ref{sec:localFortin} are element specific. 
In contrast, the results in this section apply to all inf-sup stable finite element pairs, but they do not ensure locality. 

\begin{assumption}\label{ass:Fortin_p}
Let~$(X_h,Q_h)_h$ be a family of finite element pairs as in~\eqref{def:Xh},\eqref{def:Qh}, with respect to a shape-regular family of triangulations~$(\tria_h)_h$ and let~$X_{h,0}\coloneqq X_h\cap W^{1,\infty}_0(\Omega)^d$ and~$\overline Q_h \coloneqq Q_h \cap L^\infty_0(\Omega)$. 
For~$p\in(1,\infty)$ assume that there is an operator~$\Pi_{h,0}^{p}\colon  W^{1,p}_0(\Omega)^d\to X_{h,0}$ such that 
	\begin{enumerate}[label=(\roman*)]
		\item \label{itm:Fortin_p-div}
		 the divergence is preserved in the sense that 
		\begin{align*}
			\skp{\divergence \Pi_{h,0}^{p} (v)}{q_h} = \skp{\divergence v}{q_h} \qquad \text{for all } q_h\in\overline{Q}_h,~\text{for all } v\in W^{1,p}_0(\Omega)^d;
		\end{align*}
		\item 
		\label{itm:Fortin_p-stab} 
		there is a constant $c_{F}>0$ such that 
		\begin{align*}
			\norm{\nabla \Pi_{h,0}^{p}(v)}_{L^p(\Omega)} \leq c_{F} \norm{\nabla v}_{L^p(\Omega)} \qquad \text{for all } v\in W^{1,p}_0(\Omega)^d.
		\end{align*}
	\end{enumerate} 
\end{assumption}
As before, to obtain a Fortin operator~$\Pi_h^p$ we follow the construction principle \eqref{eq:constr_inhom} as
\begin{align}\label{eq:constr_Fortin-global}
	(I-\Pi_h^p) = (I-C_h^1)(I-C_h^0)(I-I_h),
\end{align}
with an interpolation operator $I_h$ as in Lemma~\ref{lem:SZinterp}, a correction operator~$C_h^0$ handling the traces and a correction map $C_h^1$ constructed using the homogeneous Fortin operator $\Pi_{h,0}^{p}$ from Assumption~\ref{ass:Fortin_p}.
The resulting operator~$\Pi_h^p$ is in general neither local nor linear, but divergence-preservation properties and stability properties in~\ref{prop:div-pres}--\ref{prop:global-stab-approx} are inherited from~$\Pi_{h,0}^{p}$.

\begin{remark}\label{ex:globalFortin}
	\begin{enumerate}[label=(\alph*)]
		\item 
		\label{itm:example-p-inf-sup} 
		If in the setting of Assumption~\ref{ass:FEM} the discrete inf-sup condition~\eqref{def:discr-inf-sup-cond} holds for~$p\in(1,\infty)$, then the Fortin lemma~\ref{lem:Fortin_operator} yields the existence of a $p$-dependent homogeneous Fortin operator $\Pi_{h,0}^{p}$ as in Assumption~\ref{ass:Fortin_p}. 
		Note that in general this operator depends on $p$ and is not linear, and hence, any trace-preserving operator~$\Pi_h^p$ constructed from it need not be linear and may depend on $p$. 
		\item 
		\label{itm:example-p2} 
		If Assumption~\ref{ass:FEM} holds for $p=2$, then the operator $\Pi_{h,0}^{2}$ guaranteed to exist by  the Fortin lemma~\ref{lem:Fortin_operator} is linear. 
		The trace-preserving operator $\Pi_h^2$ constructed by~\eqref{eq:constr_Fortin-global}  inherits linearity and hence also the projection property~\ref{prop:lin-proj} from $I_h$, see Lemma~\ref{lem:SZinterp}. 
		\item If there exists a single homogeneous Fortin operator~$\Pi_{h,0}$ satisfying Assumption~\ref{ass:Fortin_p} for all~$p\in(1,\infty)$, then also the operator~$\Pi_h$ constructed by~\eqref{eq:constr_Fortin-global} is independent of $p$. 
		Such an operator~$\Pi_{h,0}$ can be achieved by local construction, cf. Section~\ref{sec:localFortin}.
	\end{enumerate}
\end{remark}

Let us start by constructing a Bernardi--Raugel-type operator~$C_h^0$ which only uses boundary facets, 
cf.~Lemma~\ref{lem:Ch0}.
Recall  that $	B_n^\partial(\tria_h)  = \linearspan\{b_f n_f \colon f \in \facets^{\partial}(\tria_h) \}  \subset\mathcal{L}^1_{d}(\tria_h)^d,$ denotes the space of boundary bubble functions, see~\eqref{def:bd_face_bubbles}, which is a subspace of~$B_n(\tria_h)$, see~\eqref{def:global_normal_face_bubblespace}. 
We assume that $B_n^\partial(\tria_h)\subset X_h$, or we enrich the velocity space with it, in order to work with the operator $C_h^0$.  

\begin{lemma}\label{lem:Ch0_global}
	There exists a linear operator~$C_h^0\colon W^{1,1}(\Omega)^d\to B_n^{\partial}(\tria_h)$ such that 
	\begin{enumerate}[label=(\roman*)]
		\item (trace preservation) \label{itm:Ch0-trace-global} 
 $C_h^0(W^{1,1}_0(\Omega)^d)
			\subset W^{1,1}_{0}(\Omega)^d$,  $C_h^0(W^{1,1}_n(\Omega)) \subset W_{n}^{1,1}(\Omega) $ and
		\begin{align}
			\int_f (C_h^0 v)\cdot n_f \dsig = \int_f v \cdot n_f \dsig \qquad \text{for all }f\in\facets^\partial(\tria_h), \label{eq:constrBernardi-Raugel-bd}
		\end{align}
		for any $v\in W^{1,1}(\Omega)^d$. 
		\item (stability)  \label{itm:Ch0-stab-global} for each $p \in [1,\infty]$ the operator is locally $W^{1,p}$-stable, i.e., 
		\begin{align*}
			\norm{ C_h^0 v}_{L^p(T)}	 + h_T \norm{\nabla (C_h^0 v)}_{L^p(T)} 
			\lesssim \norm{v}_{L^p(T)} +  h_T \norm{\nabla v}_{L^p(T)}, 
		\end{align*}
	 for any $v \in W^{1,p}(\Omega)^d$, for any $T \in \tria_h$, uniformly in $h$.
	\end{enumerate} 
Note that $C_h^0(v)$ is supported only in the boundary layer $\displaystyle \cup_{f \in \facets^\partial(\tria_h)}\omega_h(f)$, for any $v \in W^{1,1}(\Omega)^d$.  
The operator~$C_h^0$ is local in the sense that
	$\support(C_h^0 v) \subset \overline{\Omega_{h}(f)}$ for any $v \in W^{1,1}(\Omega)^d$ with support~$\support(v) \subset \overline{\Omega_{h}(f)}$, for any $f \in \facets^\partial(\tria_h)$.
	
\end{lemma}
\begin{proof} 
The operator is the one in Lemma~\ref{lem:Ch0} restricted to boundary normal facet bubbles. 
Then,~\eqref{eq:constrBernardi-Raugel} implies~\ref{itm:Ch0-trace-global}. 
Furthermore, the $W^{1,p}$-stability follows from~\eqref{eq:BR_Chdiv-stab}  via inverse estimates and Hölder's inequality. 
	
\end{proof}

Based on the homogeneous Fortin operator in Assumption~\ref{ass:Fortin_p} we construct a correction operator~$C_h^1$. 
Since it is based on a global operator, the scaling in the stability estimate is not optimal.

\begin{lemma}[global operator $C_h^1$]\label{lem:Chdiv_global}
	Let Assumption~\ref{ass:Fortin_p} be satisfied for some~$p \in (1,\infty)$.  
	Then, there exists an operator $	C_h^1\colon  \Wpsim\to X_{h,0}$ with the following properties 
	\begin{enumerate}[label=(\roman*)]
	\item \label{itm:Ch1_gl_div} (divergence preservation) For any $v\in \Wpsim$ one has
	\begin{align}\label{eq:div-Chdiv}
		\skp{\divergence C_h^1 (v)}{q_h} 
		= 
		\skp{\divergence v}{q_h} \qquad 
		\text{for all } q_h\in \overline{Q}_h.
	\end{align}
	\item \label{itm:Ch1_gl_stab} (global stability) There exists a constant such that 
	\begin{align}\label{eq:stab-Chdiv}
	\norm{C_h^1 v}_{W^{1,p}(\Omega)} \lesssim \norm{\nabla v}_{L^p(\Omega)} \qquad \text{for all } v\in \Wpsim,
	\end{align}
	uniformly in $h$.  
\end{enumerate}
If there exists an operator $\Pi_{h,0}^{p}$ as in Assumption~\ref{ass:Fortin_p}, which is independent of~$p \in (1, \infty)$, then so is~$C_h^1$, and if  $\Pi_{h,0}^{p}$ is linear, then $C_h^1$ is linear as well.      
	
\end{lemma}

\begin{proof}
	Since $\Omega$ is a Lipschitz domain, by Remark~\ref{rmk:Bogovskii} there exists a bounded linear operator $\mathcal{B} \colon L^p_0(\Omega) \to W^{1,p}_0(\Omega)^d$ for~$p \in (1,\infty)$  such that  
		\begin{alignat}{3}
			\divergence \mathcal{B} q  
			&= q  \qquad  
			&&\text{ for all } q \in L^p_0(\Omega), 
			\label{eq:Bog-div}
			\\
			\norm{\nabla \mathcal{B} q}_{L^p(\Omega)} 
			&\leq c \norm{q}_{L^p(\Omega)}  \qquad
			&&\text{ for all } q \in L^p_0(\Omega),
			\label{eq:Bog-stab}
		\end{alignat}	
	with constant $c$ depending only on $p$.   
	Note that the operator~$\mathcal{B}$ is independent of~$p$.  
	
	Now, with the homogeneous Fortin operator $\Pi_{h,0}^{p} \colon W^{1,p}_0(\Omega)^d \to X_{h,0}$ in Assumption~\ref{ass:Fortin_p} we set 
	\begin{align}\label{def:Ch1-global}
		C_h^1 v \coloneqq \Pi_{h,0}^{p}\, \mathcal{B}(\divergence v) \qquad \text{ for } v \in \Wpsim. 
	\end{align}
	This is well-defined since for $v \in \Wpsim$ we have $\divergence v \in L^p_0(\Omega)$, and hence $\mathcal{B}(\divergence v) \in W^{1,p}_0(\Omega)^d$. 
	
 Thanks to the divergence preservation of $\Pi_{h,0}^{p}$ in Assumption~\ref{ass:Fortin_p}~\ref{itm:Fortin_p-div} and applying~\eqref{eq:Bog-div} for any~$v\in \Wpsim$ and for any $q_h\in \overline{Q}_h$ we have that
	\begin{align*}
		\skp{\divergence C_h^1 (v)}{q_h} = \skp{\divergence \Pi_{h,0}^{p}\, \mathcal{B}(\divergence v)}{q_h} = \skp{\divergence \mathcal{B}(\divergence v)}{q_h} = \skp{\divergence v}{q_h},
	\end{align*}
	which shows~\ref{itm:Ch1_gl_div}. 
	
The stability in~\ref{itm:Ch1_gl_stab} is inherited from Assumption~\ref{ass:Fortin_p}~\ref{itm:Fortin_p-stab} and \eqref{eq:Bog-stab} as 
	\begin{align*}
		\norm{\nabla C_h^1 v}_{L^p(\Omega)} 
		&= \norm{\nabla \Pi_{h,0}^{p}\, \mathcal{B}(\divergence v)}_{L^p(\Omega)} 
		\leq
		 c_{F} \norm{\nabla \mathcal{B}(\divergence v)}_{L^p(\Omega)} 
		 \\
		&\lesssim \norm{\divergence v}_{L^p(\Omega)} 
		\lesssim \norm{\nabla v}_{L^p(\Omega)}.
	\end{align*}
By Poincaré's inequality and the fact that $C_h^1$ maps to $X_{h,0} \subset W^{1,p}_0(\Omega)^d$  we also have
 		\begin{align*}
 		\norm{C_h^1 v}_{L^p(\Omega)}\lesssim  	\norm{\nabla C_h^1 v}_{L^p(\Omega)} \lesssim \norm{\nabla v}_{L^p(\Omega)},
 	\end{align*}
 	which shows~\ref{itm:Ch1_gl_stab}. 
	By the construction in~\eqref{def:Ch1-global} the operator~$C_h^1$ inherits linearity and $p$-dependence from~$\Pi_{h,0}^{p}$, since $\mathcal{B}$ is linear and independent of $p \in (1,\infty)$.  
	This proves the claim. 
\end{proof}

Compared to the locally constructed Fortin operators in the previous section we obtain fewer properties for~$\Pi_h^p$, but the result is applicable for general inf-sup stable FEM pairs.  

\begin{proposition}\label{prop:Fortin-global}
Let Assumption~\ref{ass:FEM} be satisfied with $k \in \mathbb{N}$, and  let Assumption~\ref{ass:Fortin_p} be satisfied for some~$p \in (1,\infty)$. 
		Further, assume that $B_n^\partial(\tria_h) \subset X_h$. 
	
	 The operator~$\Pi_h^p\colon  W^{1,p}(\Omega)^d\to X_h$ defined in~\eqref{eq:constr_Fortin-global} with~$I_h$ as in Lemma~\ref{lem:SZinterp} for the space $X_h$, with~$C_h^0$ as in Lemma~\ref{lem:Ch0_global}, and with $C_h^1$ as in Lemma~\ref{lem:Chdiv_global}  is a Fortin operator satisfying the properties~\ref{prop:div-pres}, \ref{prop:trace-pres} and \ref{prop:global-stab-approx}~\ref{itm:prop_global_H1semi-stab}, \ref{itm:prop_global_Sob-stab} for the exponent~$p \in (1,\infty)$ from Assumption~\ref{ass:Fortin_p}. 
	 Furthermore, it satisfies the following version of~\ref{prop:global-stab-approx}~\ref{itm:prop_global_approx} for $1 \leq m \leq k + 1$  
		\begin{align*}
	 		\norm{v - \Pi_h^p v }_{L^p(\Omega)} +	\norm{\nabla(v - \Pi_h^p v) }_{L^p(\Omega)} \lesssim h^{m-1} \norm{\nabla^m v}_{L^p(\Omega)}
	 	\end{align*}
	 	for any $v \in W^{m,p}(\Omega)^d$, uniformly in $h$. 
	 
	 With a suitable version of $I_h$ also \ref{prop:trace-pres}$_\Gamma$ can be achieved, see Remark~\ref{rmk:properties_Fortin}~\ref{itm:rmk_part-bd}.   
	Additionally, if $\Pi_{h,0}^{p}$ as in Assumption~\ref{ass:Fortin_p} is linear, then~$\Pi_h^p$ satisfies~\ref{prop:lin-proj}. 
	If~$\Pi_{h,0}^{p}$ is independent of $p$, so is $\Pi_h^p$. 
\end{proposition}
\begin{proof}	
	The divergence property~\ref{prop:div-pres} is a consequence of the construction of $\Pi_h^p$ in~\eqref{eq:constr_Fortin-global}: 
	Indeed, for arbitrary~$v \in W^{1,p}(\Omega)^d$ let~$w \coloneqq v - I_h v$ with~$I_h$ as in Lemma~\ref{lem:SZinterp}. 
	By Lemma~\ref{lem:Ch0_global}~\ref{itm:Ch0-trace-global} it follows that~$w - C_h^0 w  \in \Wpsim$.  
	Then, Lemma~\ref{lem:Chdiv_global}~\ref{itm:Ch1_gl_div} shows that for any~$q_h \in \overline Q_h$ we have 
	\begin{align*}
		\skp{\divergence(v - \Pi_h^p v)}{q_h} 
		= 
		\skp{\divergence((I - C_h^1)(I-C_h^0) (w))}{q_h}  = 0,
	\end{align*}  
	which shows \ref{prop:div-pres}.
	
	The verification of the trace-preserving properties in~\ref{prop:trace-pres} proceeds analogously as the one in the proof of Proposition~\ref{prop:Chj_operators}, and the same is true for~\ref{prop:trace-pres}$_{\Gamma}$ thanks to the locality of $C_h^0$ in terms of boundary facets, see Lemma~\ref{lem:Ch0_global}. 
	
	It remains to establish the stability and approximation properties.  
	By the definition of $ \Pi_h^p$,  Lemma~\ref{lem:Chdiv_global}~\ref{itm:Ch1_gl_stab} for $C_h^1$, Lemma~\ref{lem:Ch0_global}~\ref{itm:Ch0-stab-global} for $C_h^0$, and the approximation properties of~$I_h$ in Lemma~\ref{lem:SZinterp}, see Proposition~\ref{prop:SZinterp}~\ref{itm:Ih-local_approx},  we obtain for $1 \leq m \leq k+ 1 $ and for any $v \in W^{m,p}(\Omega)^d$ that  
	\begin{equation}\label{est:approx-seminorm}
	\begin{aligned}
		\norm{\nabla(v - \Pi_h^p v )}_{L^p(\Omega)} 
			& \leq 
			\norm{\nabla(C_h^1(I - C_h^0)(I - I_h) v )}_{L^p(\Omega)}
			+ \norm{\nabla((I-C_h^0)(I - I_h) v)}_{L^p(\Omega)}
			\\
			& \lesssim 				\norm{\nabla((I-C_h^0)(I - I_h) v )}_{L^p(\Omega)} \\
			& \lesssim 
			\norm{\nabla(C_h^0(I - I_h) v )}_{L^p(\Omega)}	
			+ \norm{\nabla(v - I_hv)}_{L^p(\Omega)}	\\				& \lesssim 
			\norm{\tilde{h}^{-1}(v - I_hv)  }_{L^p(\Omega)}	
			+ \norm{\nabla(v - I_h v)  }_{L^p(\Omega)}	\\
			& \lesssim \norm{\tilde{h}^{m-1} \nabla^m v}_{L^p(\Omega)} 
			 \lesssim h^{m-1} \norm{ \nabla^m v}_{L^p(\Omega)},		
	\end{aligned}
	\end{equation}
	where $\tilde{h}$ denotes  the piecewise constant mesh size function with $\tilde{h}\big|_{T} = h_T$ for $T \in \tria_h$.  This shows~\ref{prop:global-stab-approx}~\ref{itm:prop_global_approx} for $\ell = 1$.  
Similarly, for the $L^p(\Omega)$-norm we obtain 
\begin{equation}\label{est:approx-Lpnorm}
	\begin{aligned}
	\norm{v - \Pi_h^p v }_{L^p(\Omega)} 
	& \leq 
	\norm{C_h^1(I - C_h^0)(I - I_h) v )}_{L^p(\Omega)}
	+ \norm{(I-C_h^0)(I - I_h) v}_{L^p(\Omega)}
	\\
	& \lesssim 				\norm{\nabla((I-C_h^0)(I - I_h) v )}_{L^p(\Omega)} + \norm{(I-C_h^0)(I - I_h) v}_{L^p(\Omega)} \\
	& 
	\lesssim h^{m-1} \norm{ \nabla^m v}_{L^p(\Omega)}.	
\end{aligned}
\end{equation}
Note that since the estimate for $C_h^1$ in $L^p(\Omega)$ is not of optimal order in $h$, this estimate is suboptimal. 
 
The stability in~\ref{prop:global-stab-approx}~\ref{itm:prop_global_H1semi-stab} follows from~\eqref{est:approx-seminorm} with $m = 1$  by using the triangle inequality.  
Combining this with the analogous estimate resulting from~\eqref{est:approx-Lpnorm} for $m = 1$ establishes the Sobolev stability in ~\ref{prop:global-stab-approx}~\ref{itm:prop_global_Sob-stab}.

Since $C_h^0$ is linear and independent of $p$, $\Pi_h^p$ inherits linearity and independence of $p$ from $C_h^1$, see Lemma~\ref{lem:Chdiv_global}. 
This proves the claim.      
\end{proof}

\begin{remark}\label{rmk:corSZviaCor}
	\begin{enumerate}
		\item 
Modifying the Scott--Zhang operator $I_h$ in Lemma~\ref{lem:SZinterp} using facet bubble functions allows us to define an operator~$\overline I_h \colon W^{1,1}(\Omega)^d \to X_h$ by
		\begin{align}\label{def:corSZviaCor}
			\overline{I}_h v \coloneqq I_h v + C_h^0 (v-I_h v),
		\end{align}	
		with $C_h^0$ as in Lemma~\ref{lem:Ch0_global}.
		 It 
		satisfies all properties of the operator~$I_h$, cf.~Proposition~\ref{prop:SZinterp}, and additionally
		\begin{align} \label{eq:compat}
			\int_{\partial\Omega} \tr(v-\overline{I}_h v)\cdot n \dsig = 0\qquad \text{for any } v\in W^{1,1}(\Omega)^d.
		\end{align}
		A similar construction is presented in~\cite[Rmk.~5.4.10]{Bernardi2024}. 
		Therein, an analogous modification of the Lagrange operator is presented for functions in~$C(\overline \Omega)^d$.  
		Note in particular that with $C_h^1$ mapping to~$W^{1,1}_0(\Omega)^d$ one has that 
		\begin{align*}
		\tr(\Pi_h^p v)  = 	\tr(\overline{I}_h v) \quad \text{ on }{\partial \Omega}.
		\end{align*}
		\item 
		An operator of the form~\eqref{def:corSZviaCor} is relevant for the implementation of inhomogeneous Dirichlet boundary conditions for exactly divergence-free finite element spaces, see~\cite{Heister2016,Eickmann2025a}.
		There, the condition~\eqref{eq:compat} is essential to ensure compatibility with the exact divergence constraints 
		for the discrete Dirichlet data.  
	\end{enumerate}
\end{remark}

\subsection*{Acknowledgements}
We thank Lars Diening and L.~Ridgway Scott for useful discussions. 

During the preparation of this work the authors used ChatGPT-5.6 Sol to assist with the literature search and to design Figure~\ref{fig:vertexpatch-reference}. The authors take full responsibility for the content of the publication.

	\printbibliography
	
\end{document}